\documentclass[letterpaper,12pt]{amsart}

\title{Unimodular Binary Hierarchical Models}

\author{Daniel Irving Bernstein}
\author{Seth Sullivant}
\address{Department of Mathematics \\ North Carolina State University, Raleigh, NC 27695}
\email{dibernst@ncsu.edu}
\email{smsulli2@ncsu.edu}

\usepackage{latexsym,array,delarray,amsthm,amssymb,epsfig}
\usepackage[numbers]{natbib}
\usepackage{tikz}
\usepackage{verbatim}
\usepackage{graphicx}

\theoremstyle{plain}
\newtheorem{thm}{Theorem}[section]
\newtheorem{lemma}[thm]{Lemma}
\newtheorem{prop}[thm]{Proposition}
\newtheorem{cor}[thm]{Corollary}

\newtheorem*{thm*}{Theorem}
\newtheorem*{lemma*}{Lemma}
\newtheorem*{prop*}{Proposition}
\newtheorem*{cor*}{Corollary}
\newtheorem*{conj*}{Conjecture}

\theoremstyle{definition}
\newtheorem{defn}[thm]{Definition}
\newtheorem*{defn*}{Definition}
\newtheorem{ex}[thm]{Example}

\newtheorem{ques}[thm]{Question}

\theoremstyle{remark}

\newcommand{\zz}{\mathbb{Z}}
\newcommand{\nn}{\mathbb{N}}

\newcommand{\rr}{\mathbb{R}}

\newcommand{\kk}{\mathbb{K}}

\newcommand{\bfd}{\mathbf{d}}

\newcommand{\bfi}{\mathbf{i}}
\newcommand{\bfj}{\mathbf{j}}
\newcommand{\bfk}{\mathbf{k}}

\newcommand{\calc}{\mathcal{C}}

\newcommand{\ind}{\mbox{$\perp \kern-5.5pt \perp$}}

\newcommand{\rank}{\textnormal{rank}}
\newcommand{\link}{\textnormal{link}}

\newcommand{\cone}{\textnormal{cone}}
\newcommand{\facet}{\textnormal{facet}}

\tikzstyle{vertex}=[circle, draw, inner sep=0pt, minimum size=6pt, fill=black]
\newcommand{\vertex}{\node[vertex]}

\begin{document}

\begin{abstract}
Associated to each simplicial complex is a binary hierarchical model.
We classify the simplicial complexes that yield unimodular
binary hierarchical models.  Our main theorem provides
both a construction of all unimodular binary hierarchical models, together
with a characterization in terms of excluded minors,
{where our definition of a minor allows the taking of links and induced complexes.}
A key tool in the
proof is the lemma that the class of unimodular binary hierarchical models is closed
under the Alexander duality operation on simplicial complexes.
\end{abstract}

\maketitle


\section{Introduction}\label{introduction}

Associated to a simplicial complex $\calc$ with ground set $[m]$
and an integer vector $\bfd \in \zz^m$ is an integral matrix
$\mathcal{A}_{\calc, \bfd}$.  The hierarchical model associated to
$\calc, \bfd$ is a log-linear model (i.e.~discrete exponential family) whose design matrix is the matrix $\mathcal{A}_{\calc, \bfd}$.
{ The Zariski closure of a hierarchical model is a (not necessarily normal) toric variety.}
As for all log-linear models, important relevant problems
are to study properties of the lattice $\ker_\zz \mathcal{A}_{\calc, \bfd}$,
the polyhedral cone 
$\rr_{\geq 0} \mathcal{A}_{\calc, \bfd} := \{\mathcal{A}_{\calc, \bfd}x : x \geq 0 \} $ 
and the semigroup
$\nn \mathcal{A}_{\calc, \bfd} :=  \{\mathcal{A}_{\calc, \bfd}x : x \geq 0, x \mbox{ integral } \} $.  For example, Markov bases of the lattice $\ker_\zz \mathcal{A}_{\calc, \bfd}$
can be used for goodness-of-fit tests of the log-linear model
associated to $\mathcal{A}_{\calc, \bfd}$ (see \cite{Diaconis1998}).
In the special case where $\bfd = {\bf 2}$ the vector of all twos,
 we call $\mathcal{A}_{\calc, {\bf 2}}$ a
binary hierarchical model.  

There are a number of ``niceness'' properties that an integral matrix $A \in \zz^{d \times n}$ 
could satisfy.  Perhaps the strongest is unimodularity.  
\begin{defn}\label{def:unimodular}
The following are equivalent definitions for an integral matrix $A$ to be \emph{unimodular}.
\begin{enumerate}
\item  The polyhedron $P_{A,b} = \{x \in \rr^s : Ax = b, x \ge 0\}$
 has all integral vertices for every $b \in \nn A$.
\item  Suppose $\rank A = d$.  Then there exists a $\lambda$ such that
 the determinants of all $d \times d$ submatrices of $A$ are $0$ or $\pm \lambda$.
\item  The entries in each circuit of $A$ are either $0$ or $\pm 1$.
\item  The Graver basis of $A$ consists of elements whose entries
are either $0$ or $\pm 1$.
\end{enumerate}
\end{defn}
We elaborate on this definition in Section \ref{sec:preliminaries}.
A weaker property for the matrix $A$ to satisfy is normality.
\begin{defn}
A matrix $A$ is \emph{normal} if the semigroup $\nn A$ is saturated, that is
$\nn A  =  \zz A  \cap  \rr_{\geq 0} A$.
\end{defn}

In fact, every unimodular matrix $A$ is normal.  In previous work
of Rauh and the second author \cite{Rauh2014}, normality was identified as a key
property of hierarchical models to be able to apply the toric fiber product  
construction to calculate a Markov basis.  {If $A_{\calc, \bfd}$
is unimodular, it is also easy to solve the integer programs that arise
in sequential importance sampling \cite{Chen2006}.}
  It is a major open problem to classify the 
normal hierarchical models.  The special case where the underlying simplicial complex
$\calc$ is a graph, and where $\bfd = {\bf 2}$ was
handled in \cite{Sullivant2010} where it was shown that a binary graph
model is normal if and only if the graph is free of $K_4$ minors.

Our goal in the present paper is to classify the unimodular binary
hierarchical models.  In fact, this result is also an important first
step in the classification of normal binary hierarchical models because
a matrix $\Lambda(A)$ of Lawrence-type is normal
if and only if unimodular (this follows from \cite{indispensable} Corollary 16 in light of the circuit definition of unimodularity).
So this provides a characterization of the normal binary hierarchical models that have a big facet
(see Definition \ref{def:lawrence} and remarks that follow).

Our main result is a complete structural characterization of the 
unimodular binary hierarchical models.  On the one hand, we show that
every unimodular binary hierarchical model can be built up
from three basic families, together with three operations.
As part of that proof we also give an excluded minor characterization
of the unimodular binary hierarchical models.  Here our notion of minor
of a simplicial complex is one that is obtained by taking a sequence of
vertex deletions or links of vertices.
{We caution the reader that our notion of simplicial complex minor
does not directly generalize the notion of a graph minor, but it does generalize the notion of a matroid minor.}
Our excluded minor characterization
involves an infinite number of minimally nonunimodular complexes.

Our characterization of unimodular binary hierarchical models can be
seen in analogy with the classic characterizations of 
regular matroids/totally unimodular matrices, which have both
excluded minor characterizations and constructive characterizations.
{While there is clearly a close connection between our problem and the
study of totally unimodular matrices, our methods do not make use of
Seymour's \cite{Seymour1980} nor Tutte's \cite{tutte1958} classification of regular matroids.
There are two reasons for this.
The first is that the set of matroids underlying the design matrices of hierarchical models
is not closed under the operations of direct sum, two sum, three sum, deletion, nor contraction.
The second is that normality, unlike (total) unimodularity, is not a matroidal property
and we are interested in extending our methods to the normal classification problem \cite{bernstein-sullivant2015}.
}

The outline of this paper is as follows.  In the next section,
we provide a detailed introduction to hierarchical models and the construction
and interpretation of the matrix $\mathcal{A}_{\calc, \bfd}$.  
Section \ref{sec:basic} details basic examples of unimodular complexes
and constructions of new unimodular complexes from old ones.
A main result here is Proposition \ref{dual_uni} which shows that Alexander duality
preserves unimodularity.
Section \ref{sec:bavoid} gives the list of minimal non-unimodular binary hierarchical
models.  There are one infinite family and six sporadic minimal non-unimodular complexes.
The section also provides a number of structural results on simplicial complexes
that avoid this list of simplicial complexes.  Section \ref{sec:1skel} provides
the combinatorial description of the $1$-skeleton of a simplicial
complex that supports a unimodular binary hierarchical model.  That result
is used as a stepping stone in a complex induction in Section \ref{sec:main}
which gives the proof of the main theorem, Theorem \ref{main}.
In Section \ref{sec:non} we explain how our characterization in the binary case
can be used to make headway in the non-binary case.


\section{Preliminaries on Hierarchical Models
and Unimodularity}\label{sec:preliminaries}

In this section we explain how to construct the matrix $\mathcal{A}_{\calc, \bfd}$ associated to a hierarchical model.
We then define some terms from Definition \ref{def:unimodular} of a unimodular matrix  and prove the equivalence between the four parts.

\begin{defn}
	Let ${\bf d} = (d_1,\dots,d_n) \in \zz^n$ be such that $d_i \ge 2$ for each $i$.
	Then we define $\rr^{\bf d}$ to be the vector space of all $d_1\times \dots \times d_n$-way tables.
	That is, each ${\bf u} \in \rr^{\bf d}$ is a table in $n$ dimensions where the $i$th dimension has $d_i$ levels.
	We denote entries in a table ${\bf u}$ by $u_{i_1,\dots,i_n}$ where $(i_1,\dots,i_n) \in [d_1]\times\dots\times [d_n]$.
	We will use the shorthand ${\bf i}$ to denote ${i_1,\dots,i_n}$.
\end{defn}

\begin{defn}
	Let $\mathcal{C}$ be a simplicial complex on ground set $[n]$.
	A \emph{facet} of $\mathcal{C}$ is an inclusion-maximal subset $F \subseteq [n]$
	that is contained in $\mathcal{C}$.
	We let $\facet(C)$ denote the collection of facets of $\mathcal{C}$.
\end{defn}

\begin{defn}
	Let ${\bf d} = (d_1,\dots,d_n) \in \zz^n$ and $F = \{f_1,\dots,f_k\} \subseteq [n]$.
	We define ${\bf d}_F$ to be the restriction of ${\bf d}$ to the indices in $F$,
	i.e. ${\bf d}_F = (d_{f_1},\dots,d_{f_k})$.
\end{defn}

\begin{defn}
	Let $\mathcal{C}$ be a simplicial complex on ground set $[n]$
	and let ${\bf d} = (d_1,\dots,d_n) \in \zz^n$ be such that $d_i \ge 2$ for each $i$.
	Then we have the linear map
	\[
		\pi_{\calc,{\bf d}}: \rr^{\bf d} \rightarrow \bigoplus_{F \in \facet(C)} \rr^{{\bf d}_F}
	\]
	defined by
	\[
		(u_{i_1, \ldots, i_n}  :  i_j \in [d_j])  \mapsto  \bigoplus_{\substack{{F \in {\rm facet}(\calc)}\\F=\{f_1,\dots,f_k\}}}
		\left( \sum_{\substack{{\bf i} \in [d_1]\times\dots\times[d_n]:
				\\{\bf i}_F = {\bf j}}}  u_{i_1, \ldots, i_n} : {\bf j} \in [d_{f_1}]\times\dots\times[d_{f_k}]  \right).
	\]

	The matrix $\mathcal{A}_{\calc,{\bf d}}$ denotes the matrix in the standard basis
	that represents the linear transformation $\pi_{\mathcal{C},{\bf d}}$.
\end{defn}

We give an example of this construction.

\begin{ex}		
Let $\calc$ be the simplicial complex on ground set $[3] = \{1,2,3\}$ with
${\rm facet}(\calc) =  \{ \{1\}, \{2,3\} \}$.
The linear transformation $\pi_{\calc, \bfd}$
maps a three-way tensor $u = (u_{ijk} : i \in [d_1], j \in [d_2], k \in [d_3] ) \in \rr^D$
to the direct sum of a one-way tensor and a two-way tensor:
$$
\pi_{\calc, \bfd}( u)  =    ( \sum_{j,k}  u_{ijk} :  i \in [d_1])  \oplus 
( \sum_{i}  u_{ijk} :  j \in [d_2], k \in [d_3]). 
$$
Taking $\bfd = {\bf 2}$, the matrix $\mathcal{A}_{\calc, {\bf 2}}$ could be represented
as follows:
$$
\begin{pmatrix}
1 & 1 & 1 & 1 & 0 & 0 & 0 & 0 \\
0 & 0 & 0 & 0 & 1 & 1 & 1 & 1 \\
\hline
1 & 0 & 0 & 0 & 1 & 0 & 0 & 0 \\
0 & 1 & 0 & 0 & 0 & 1 & 0 & 0  \\
0 & 0 & 1 & 0 & 0 & 0 & 1 & 0 \\
0 & 0 & 0 & 1 & 0 & 0 & 0 & 1  \\
\end{pmatrix},
$$
the horizontal line dividing the image space into the parts corresponding
to the two facets of $\calc$.
\end{ex}

It can be useful to be precise about how to write down the matrix
$\mathcal{A}_{\calc, \bfd}$.  The columns of the matrix
$\mathcal{A}_{\calc, \bfd}$ are indexed by elements of the set
$\prod_{j \in [n]}  [d_j]$ and the rows are indexed by pairs
$(F, e)$ where $F \in {\rm facet}(\calc)$ and $e \in \prod_{j \in F}  [d_j]$.
The entry in position with row index $(F, e)$ and column indexed by $i \in \prod_{j \in [n]}  [d_j]$ is $1$ if $e  =  (i_j :  j \in F)$, otherwise it is equal to
$0$.

The binary case studied in this paper concerns the situation where
all table dimensions are of size $2$, that is $\bfd = {\bf 2} = (2,2,\dots,2)$; the vector of all twos.
In this case, we will usually abbreviate notation as stated in the following definition.

\begin{defn}
	Let $\mathcal{C}$ be a simplicial complex on $[n]$.
	We define $\mathcal{A}_{\calc} := \mathcal{A}_{\calc,{\bf 2}}$.
	We say that $\mathcal{C}$ is \emph{unimodular} if $\mathcal{A}_\calc$ is unimodular as a matrix.
\end{defn}

We now elaborate on Definition \ref{def:unimodular}.
To start, we give two definitions.
\begin{defn}\label{def:circuit}
	Let $A$ be an integer matrix.
	A nonzero element $u \in \ker_\zz A$
	is called a \emph{circuit} if its nonzero entries are relatively prime
	and there is no other nonzero element $v \in \ker_\zz A$ such that ${\rm supp}(v) \subset {\rm supp}(u)$.
 	Here, ${\rm supp}(v)$ denotes the set of  positions of nonzero elements of $v$.
\end{defn}

\begin{defn}\label{defn:graver}
	Let $A$ be an integer matrix.
	A nonzero element $u \in \ker_\zz A$ is called \emph{primitive} if there is no nonzero
	$v \in \ker_\zz A$, $v \neq u$ such that $u_i v_i \geq 0$ for all $i$, and $|v_i| \leq |u_i|$ for all $i$.
	The set of all primitive vectors in $\ker_\zz A$ is called the \emph{Graver basis} of $A$.
\end{defn}

Now we explain the equivalence between the four items in Definition \ref{def:unimodular}.
See \cite{Schrijver1986}, Theorem 19.2 for the equivalence of (1) and (2).
We can see that condition (3) of Definition \ref{def:unimodular} is equivalent to
condition (2) because the circuits can be computed via a determinantal formula
using Cramer's rule \cite[p.~35]{sturmfels}.
Given the equivalence of (2) and (3),
(3) and (4) are equivalent by \cite{sturmfels}, Propositions 4.11 and 8.11.

Questions about unimodularity of a matrix $A$ are invariant under certain changes to the matrix.
The following proposition gives three that we make tacit use of throughout the rest of the paper.
 
 \begin{prop}\label{obvious}
 	Let $A$ be a unimodular integer matrix.
 	Then the following are also unimodular matrices.
	\begin{enumerate}
		\item\label{item:reorder} A matrix obtained by reordering the columns of $A$
		\item\label{item:negative} A matrix obtained by multiplying a column of $A$ by $-1$
		\item\label{item:rowspace} Any matrix that has the same rowspace as $A$
		\item\label{item:subset} Any matrix formed by a subset of columns of $A$.
	\end{enumerate}
 \end{prop}
 \begin{proof}
 	Unimodularity of (\ref{item:reorder}) and (\ref{item:rowspace}) are clear when considering (3) or (4) in Definition \ref{def:unimodular}.
	Unimodularity of (\ref{item:negative}) is clear when considering (2) in Definition \ref{def:unimodular}.
	If $A'$ is a matrix obtained by extracting a set of columns of $A$
	then the circuits of $A'$ are obtained by taking those circuits
	of $A$ whose nonzero elements correspond to the columns of $A$ we are extracting,
	and so (\ref{item:subset}) holds. 
 \end{proof}

\section{Constructions of Unimodular Complexes}\label{sec:basic}

In this section we describe some operations on simplicial complexes 
that preserve unimodularity.
In particular we show that unimodularity is preserved when 
passing to induced subcomplexes, to the Alexander dual of a simplicial complex,
and to the link of a face of a simplicial complex.  We can also
build new unimodular complexes from old ones by adding cone vertices,
ghost vertices, or taking a Lawrence lifting.
We also construct the basic examples of unimodular complexes.
These tools together go in to our constructive description of 
unimodular complexes.

\begin{prop}\label{induced}
	All induced sub-complexes of a unimodular complex are unimodular.
\end{prop}

\begin{proof}
	Let $\calc'$ be the induced subcomplex of $\calc$ obtained by restricting to some vertex set $F \subseteq [n]$.
	Let $B$ be the matrix obtained by taking the columns of $\mathcal{A}_{\calc}$
	corresponding to the elements  $  i = (i_1, \ldots, i_n)  \in\prod_{j \in [n]} [d_j]$ such
	that $i_j = 1$ for all $j \in [n] \setminus F$.
	{Let $B'$ be the matrix obtained from $B$ by removing rows of all zeros.
	Then $B' = \mathcal{A}_{\mathcal{C}'}$.
	So $\mathcal{A}_{\mathcal{C}'}$ and $B$ have the same Graver basis,
	which is a subset of the Graver basis of $\mathcal{A}_\calc$.
	So unimodularity of $\mathcal{A}_\calc$ implies unimodularity of $\mathcal{A}_{\calc'}$.}
\end{proof}

Next we will show that unimodularity is preserved under the Alexander duality
operation.

\begin{defn}
	Let $\mathcal{C}$ be a simplicial complex on $[n]$.  The \emph{Alexander dual} of $\mathcal{C}$, denoted $\calc^*$ is defined as:
	\[
		\mathcal{C}^*=\{S\subseteq [n] : [n]\setminus S \notin \mathcal{C}\}.
	\]
\end{defn}

Note that if $\mathcal{C}$ has $d$ faces, then $\mathcal{C}^*$ has $2^n-d$ faces
and that the facets of $\mathcal{C}^*$ are the complements of the \emph{minimal} non-faces of $\mathcal{C}$.

{
Since we are now restricting attention to the binary case, where $d_i = 2$ for each $i \in [n]$,
entries of vectors in $\rr^{\bfd}$ are indexed by elements of $\{1,2\}^n$
and entries of vectors in $\bigoplus_{f \in \facet(\mathcal{C})} \rr^{\bfd_F}$ are indexed by pairs $(S,\bfj)$
where $S \in \facet(\mathcal{C})$ and $\bfj \in \{1,2\}^{S}$.
However, the statement and proof of Proposition \ref{spanningset} become cleaner
if we replace $\{1,2\}^n$ and $\{1,2\}^S$ by $\{0,1\}^n$ and $\{0,1\}^S$, respectively.
So we adapt this convention for the remainder of this section.
We now introduce some notation that we use in Propositions \ref{dual_matrix} and \ref{spanningset}.

\begin{defn}
	Let $S \subseteq [n]$ and $\bfj \in \{0,1\}^S$ and $\bfi \in \{0,1\}^{[n]\setminus S}$.
	Then $e_{\bfi,\bfj}$ denotes the element $v \in \rr^{\{0,1\}^n}$ such that $v_{\bfk} = 1$ if $k_t = j_t$
	when $t \in S$ and $k_t = i_t$ when $t \in [n]\setminus S$, and $v_{\bfk} = 0$ otherwise.
\end{defn}
}

We are now ready to give an explicit description of $\mathcal{A}_{\mathcal{C}^*}$.

\begin{prop}\label{dual_matrix}
	Let $\mathcal{C}$ be a simplicial complex on $[n]$.
	Let $M$ be a matrix with the following set of columns:
	\[
		\left\{\sum_{\bfj \in \{0,1\}^S} e_{\bfi,\bfj} : S \textnormal{ is a minimal non-face of } \mathcal{C}, \bfi \in \{0,1\}^{[n]\setminus S}\right\}.
	\]
	Then $\mathcal{A}_{\mathcal{C}^*} = M^T$.
\end{prop}

\begin{proof}
	We index the columns of $M$ by pairs $(S,\bfi)$ where $S$ is a minimal non-face of $\mathcal{C}$, and $\bfi \in \{0,1\}^{[n]\setminus S}$.
	The rows of $M$ are indexed by elements of $\{0,1\}^n$.
	By construction, $M(\bfk,(S,\bfi))=1$ if and only if $\bfk|_{[n]\setminus S} = \bfi$.
	Similarly, when $F$ is a facet of $\mathcal{C}^*$, $\mathcal{A}_{\mathcal{C}^*}((F,\bfi),\bfk) = 1$ if and only $\bfk|_{F}=\bfi$.
	The result follows since $S$ is a minimal non-face of $\mathcal{C}$ if and only if $[n]\setminus S$ is a facet of $\mathcal{C}^*$.
\end{proof}

In order to relate the unimodularity of $\mathcal{A}_\mathcal{C}$ and $\mathcal{A}_{\mathcal{C}^*}$, we need two propositions.  The first
is a standard result from the theory of matroid duality, {though we could not find
a precise reference}.

\begin{prop}\label{dualdet}
	Let $A \in \kk^{r\times n}$ have rank $r$.
	Let $B \in \kk^{(n-r)\times n}$ have rank $n-r$ such that $AB^T=0$.
	Then there exists a non-zero scalar $\lambda \in \kk^*$ such that for all $S\subseteq[n]$ with $\#S=r$,
	\[\det(A_S)= \pm\lambda \det(B_{[n]\setminus S}).\]
\end{prop}

{
\begin{proof}
After multiplying by an element of $GL_r(\kk)$ and permuting columns, we can assume
$A$ has the form $A = ( I_r \, \, M)$ where $I_r$ is an $r \times r$ identity matrix.  Then
we can suppose $B$ has the form  $B = (-M^T \, \,  I_{n-r})$.  Then any $\det(A_S)$ is a minor
of $M$ with row indices $[r] \setminus S$ and column indices $S \setminus [r]$.
The determinant $\det(B_{[n] \setminus S})$ is the same minor of $M$, up to a sign. 
\end{proof}
}

\begin{prop}\label{spanningset}
	Let $\mathcal{C}$ be a simplicial complex on $[n]$.
	Then the following set of vectors spans $\ker(\mathcal{A}_\mathcal{C})$:
	\[
		K=\left\{\sum_{\bfj \in \{0,1\}^S} (-1)^{\|\bfj\|_1} e_{\bfi,\bfj} : S \textnormal{ is a minimal non-face of } \mathcal{C}, \bfi \in \{0,1\}^{[n]\setminus S}\right\}.
	\]
	Furthermore, if we view the entries of $K$ as the columns of a matrix $M'$, then we can multiply some set of
	rows and columns of $M'$ by $-1$ and get $M$, as defined in Proposition \ref{dual_matrix}.
\end{prop}

\begin{proof}
	First, note that for any $x \in K$, $\mathcal{A}_\mathcal{C} x=0$, so $K\subset \ker(\mathcal{A}_\mathcal{C})$.
	If we view the entries of $K$ as the columns of a matrix $M'$, we are done if we show that $\rank(M') = \dim(\ker(\mathcal{A}_\mathcal{C}))$.
	We proceed by proving the second statement of the lemma.
	The first statement will follow since Proposition \ref{dual_matrix} implies $\rank(M)=\rank(\mathcal{A}_{\mathcal{C}^*})$ and we know
	$\rank(\mathcal{A}_{\mathcal{C}^*})= 2^n-\#\mathcal{C}=\dim(\ker(\mathcal{A}_\mathcal{C}))$ (\cite[Thm 2.6]{Hosten2002}).
	\\
	\indent
	As before, we index the columns of $M'$ by pairs $(S,\bfi)$ where $S$ is a minimal non-face of $\mathcal{C}$,
	and we index the rows of $M'$ by the binary $n$-tuples.
	Let $M'_{(S,\bfi)}$ be a column of $M'$.
	For any $\bfk \in \{0,1\}^n$, the entry $M'_{(S,\bfi),\bfk}$ is nonzero if and only if $\bfk|_{[n]\setminus S} = \bfi$ - i.e. for each column, the $\bfk$ for
	each nonzero entry has fixed $\bfk|_{[n]\setminus S} = \bfi$.
	Now in this case, recall that $M'_{(S,\bfi),\bfk} = +1$ if $\|\bfj\|_1$ is even, and $M'_{(S,\bfi),\bfk} = -1$ if $\|\bfj\|_1$ is odd.
	Therefore, when $\|\bfi\|_1$ is even, $M'_{(S,\bfi),\bfk} = +1$ if $\|\bfk\|_1$ is even and $M'_{(S,\bfi),\bfk} = -1$ when $\|\bfk\|_1$ is odd.
	And when $\|\bfi\|_1$ is odd, $M'_{(S,\bfi),\bfk} = +1$ when $\|\bfk\|_1$ is odd and $M'_{(S,\bfi),\bfk} = -1$ when $\|\bfk\|_1$ is even.
	So for a fixed column $M'_{(S,\bfi)}$, the sign of a nonzero entry $M'_{(S,\bfi),\bfk}$ depends only on the parity of $\bfk$.
	So if we multiply all the rows $M'_{\cdot,\bfk}$ such that $\|\bfk\|_1$ is odd by $-1$, the entries in each column will all have the same sign.
	Then we can multiply all the negative columns by $-1$, to arrive at the matrix
	from Proposition \ref{dual_matrix}.
\end{proof}

Now we are ready to show that Alexander duality preserves unimodularity.

\begin{prop}\label{dual_uni}
	Let $\mathcal{C}$ be a simplicial complex on ground set $[n]$.
	Then $\mathcal{C}$ is unimodular if and only if $\mathcal{C}^*$ is unimodular.
\end{prop}
\begin{proof}
	Consider the set $K$ from Proposition \ref{spanningset} as a matrix of column vectors.
	Since $K$ spans the kernel of $\mathcal{A}_\mathcal{C}$, Proposition \ref{dualdet} implies that $K^T$ is unimodular if and only if $\mathcal{A}_\mathcal{C}$ is unimodular.
	Then we can multiply the appropriate rows and columns of $K$ by $-1$ to get $M$ as in Proposition \ref{dual_matrix}.
	Since these operations preserve the absolute values of full rank determinants, $K^T$ is unimodular if and only if $M^T$ is.
	Proposition \ref{dual_matrix} says that $\mathcal{A}_{\mathcal{C}^*} = M^T$.
\end{proof}

Taking Alexander duals and induced complexes gives rise to another unimodularity preserving operation which we now define.
\begin{defn}
	Let $S  \in \mathcal{C}$ be a face of $\calc$.
	Then the link of $S$ in $\mathcal{C}$ is the new simplicial complex
	\[
		\link_S( \mathcal{C} ) = \left\{F\setminus S : F \in \mathcal{C} \mbox{ and } S \subseteq F \right\}.
	\]
	When $S = \{v\}$, we simply write $\link_v(\mathcal{C}) := \link_{\{v\}}( \calc)$.
\end{defn}

Note that we can obtain $\link_S(\mathcal{C})$ by repeatedly taking links with respect to vertices.
That is if $S$ is a face of $\calc$ and  $\#S \ge 2$ and $v \in S$, then
\[
	\link_S(\mathcal{C}) = \link_{v}(\link_{S\setminus\{v\}}(\mathcal{C})).
\]

\begin{prop}\label{linkdual}
	If $\mathcal{C}$ is a simplicial complex and $S$ is a face of $\calc$, 
	then $\link_S(\mathcal{C}) = (\mathcal{C}^*\setminus S)^*$. \qedhere
\end{prop}

\begin{proof}
	By definition, we have:
	\[
		(\mathcal{C}^*\setminus S)^* = \left\{R \subseteq [n] \setminus S : \left([n]\setminus S\right) \setminus R \notin \mathcal{C}^*\setminus S \right\}.
	\]
	Then we have the following chain of equivalences on some $R \subseteq [n] \setminus S$:
	\begin{align*}
		&([n] \setminus S)\setminus R \notin \mathcal{C}^*\setminus S
		\\&\iff ([n] \setminus S)\setminus R \notin \mathcal{C}^*
		\\&\iff [n] \setminus \left(([n] \setminus S)\setminus R\right) \in \mathcal{C}
	\end{align*}
	and since $[n] \setminus \left(([n] \setminus S)\setminus R\right) = R \cup S$, we have $R \cup S \in \mathcal{C}$.
\end{proof}

\begin{cor}\label{link}
	Let $\mathcal{C}$ be a unimodular simplicial complex on ground set $[n]$.
	Then for any face $S$ of $\calc$, $\link_S( \mathcal{C} )$ is unimodular.
\end{cor}

\begin{proof}
    Proposition \ref{linkdual} implies that $\link_S( \mathcal{C} )$  can be obtained
    via Alexander duality and passing to an induced subcomplex.
	Unimodularity then follows from Propositions \ref{induced} and \ref{dual_uni}.
\end{proof}

Now we turn to operations for taking a complex that
is unimodular and constructing larger unimodular complexes.
If $\mathcal{C}$ has a vertex $v$ that lies in each facet of $\mathcal{C}$, we say that $v$ is a \emph{cone vertex} of $\mathcal{C}$.
The following proposition tells us that unimodularity is invariant under adding or removing a cone vertex.

\begin{prop}\label{cone}
	Let $\mathcal{C}$ be a simplicial complex on $[n]$.
	We define $\mathcal{C}'$ on $[n+1]$ to be the simplicial complex with the following facets:
	\[
		\{F \cup \{n+1\}: F \textnormal{ is a facet of } \mathcal{C}\}.
	\]
	Then $\mathcal{A}_\mathcal{C}$ is unimodular if and only if $\mathcal{A}_{\mathcal{C}'}$ is.
\end{prop}

\begin{proof}
	We will index the columns of $\mathcal{A}_{\mathcal{C}'}$ by the binary $n+1$ tuples
	such that those with the $n+1$ coordinate equal to 1 come before those with $n+1$ coordinate equal to $2$.
	Then we get the following block form:
	\[
		\mathcal{A}_{\mathcal{C}'} = \left(\begin{array}{c|c}
		\mathcal{A}_\mathcal{C} & {\bf 0} \\ \hline
		{\bf 0} & \mathcal{A}_\mathcal{C}
		\end{array}\right).
	\]
	The Graver basis of $\mathcal{A}_{\mathcal{C}'}$ is therefore $\{(u,0), (0,u) : u \in Gr_A\}$.
	Hence the Graver basis of $\mathcal{A}_{\calc'}$ consists of $0, \pm 1$ elements if and
	only if the Graver basis of $\mathcal{A}_\calc$ consists of $0, \pm 1$ elements.
\end{proof}

By induction, adding or removing multiple cone vertices from a simplicial complex does not affect unimodularity.
We introduce the following notation to denote this.

\begin{defn}\label{cone}
	Let $\mathcal{C}$ be a simplicial complex on vertex set $[n]$. Then we define $\cone^p(\mathcal{C})$
	to be the simplicial complex on $[n] \cup \{u_1,\dots,u_p\}$ with the following facets
	\[
		\{F \cup \{u_1,\dots,u_p\} : F \text{ is a facet of } \mathcal{C} \}.
	\]
	When $p = 0$, we define $\cone^p(\mathcal{C}) = \mathcal{C}$.
\end{defn}

\begin{defn}
	For any matrix $A \in \rr^{s \times t}$, we define the \emph{Lawrence lifting} of $A$ to be the matrix
	\[
		\Lambda(A) \quad =  \quad
			\begin{pmatrix}
				A & {\bf 0}\\
				{\bf 0} & A\\
				{\bf 1} & {\bf 1}
			\end{pmatrix}  \quad \in \quad   \rr^{ (2s + t)  \times 2t}
	\]
	where ${\bf 0}$ denotes the $s \times t$ matrix of all zeroes and 
	${\bf 1}$ denotes a $t \times t$ identity matrix.
\end{defn}

By Theorem 7.1 in \cite{sturmfels}, $\Lambda(A)$ is unimodular if and only if $A$ is unimodular.
This gives rise to another unimodularity-preserving operation on simplicial complexes.

\begin{defn}\label{def:lawrence}
	Let $\mathcal{C}$ be a simplicial complex on $[n]$.
	We define the \emph{Lawrence lifting} of $\mathcal{C}$
	to be the simplicial complex $\Lambda(\mathcal{C})$ on $[n+1]$ that has the following set of facets:
	\[
		\{ [n] \} \cup \{F \cup \{n+1\} : F \textnormal{ is a facet of } \mathcal{C}\}.
	\]
	In this case, we refer to the facet $[n]$ as a \emph{big facet}.
\end{defn}

If a complex $\mathcal{C}$ on $[n+1]$ has big facet $[n]$,
then $\mathcal{C} = \Lambda(\link_{n+1}(\mathcal{C}))$.
In particular, $\mathcal{C} = \Lambda(\mathcal{D})$ for some complex $\mathcal{D}$
if and only if $\mathcal{C}$ has a big facet.

\begin{prop}\label{lawrence}
	The simplicial complex $\mathcal{C}$ is unimodular if and only if $\Lambda(\mathcal{C})$ is unimodular.
\end{prop}

\begin{proof}
	Note that by construction $\Lambda(\mathcal{A}_{\mathcal{C}}) = \mathcal{A}_{\Lambda(\mathcal{C})}$.
	Hence Theorem 7.1 in \cite{sturmfels} implies the proposition.
\end{proof}

If a simplicial complex $\mathcal{C}$ can be expressed as $\Lambda(\mathcal{C'})$,
then we say that $\mathcal{C}$ is of Lawrence type.
A simplicial complex on $n$ vertices is of Lawrence type if and only if it has a facet containing $n-1$ vertices.  We will refer to any facet of $\calc$ that has $n-1$ elements as a \emph{big facet}.

We now give our final unimodularity preserving operation.

\begin{defn}
	Let $\mathcal{C}$ be a simplicial complex on ground set $[n]$.
	Let $G^p(\mathcal{C})$ denote the same simplicial complex but on ground set $[n+p]$.
	Note that the vertices $n+1,\dots,n+p$ are not contained in any face of $G^p(\mathcal{C})$.
	In this case we say that $n+1,\dots,n+p$ are \emph{ghost vertices}.
	When $p = 1$ we drop the superscript; i.e. we just write $G(\mathcal{C})$.
\end{defn}

\begin{prop}\label{ghost}
	A simplicial complex $\mathcal{C}$ is unimodular if and only if $G^p(\mathcal{C})$ is unimodular.
\end{prop}
\begin{proof}
	This is true because
	\[
		\mathcal{A}_{G^p(\mathcal{C})} = \begin{pmatrix}\mathcal{A}_{\mathcal{C}} & \dots & \mathcal{A}_{\mathcal{C}}\end{pmatrix}. \qedhere
	\]
\end{proof}

Note that adding a ghost vertex to a complex is Alexander dual to
taking the Lawrence lifting of a simplicial complex, that is
$G(\calc)^{*}  =  \Lambda(\calc^{*})$. 

We now state a useful fact about the interaction of these operations.
\begin{prop}\label{commutes}
	The operation $\cone^p(\cdot)$ commutes with the operations of taking links, duals, induced sub-complexes and adding ghost vertices.
\end{prop}

For any vertex $v$ of $\mathcal{C}$, we let $\mathcal{C}\setminus v$ denote the induced subcomplex on $\mathcal{C}\setminus\{v\}$.
Then for vertices $v\neq u$ of $\mathcal{C}$, the operations $\cdot\setminus v$ and $\link_u(\cdot)$ commute.
So if $\mathcal{D}$ can be obtained from $\mathcal{C}$ by applying a series of deletions, and taking links,
then we can write
\[
	\mathcal{D} = \link_R(\mathcal{C}\setminus S)
\]
where $F \in \mathcal{C}$ and $S,R$ are a subset of vertices of $\mathcal{C}$ such that $S \cap R = \emptyset$.
This gives rise to the following definition:

\begin{defn}\label{def:minor}
	Let $\mathcal{C}$ be a simplicial complex. Then $\mathcal{D}$ is a \emph{minor} of $\mathcal{C}$
	if $\mathcal{D} = \link_R(\mathcal{C}\setminus S)$ where
	and $S,R$ are subsets of vertices of $\mathcal{C}$ such that $S \cap R = \emptyset$,
	and $R$ is a face of $\calc$.
\end{defn}

{ We warn the reader that although it is natural to view simplicial complexes as a generalization of graphs,
our definition of simplicial complex minor does \emph{not} generalize the usual notion of graph minor.
Alternatively, a simplicial complex can be seen as a generalization of a matroid,
and our definition of minor is a generalization of the usual notion of 
a matroid minor \cite{kaiser2009}.
This definition of simplicial complex minor is useful for our 
purposes because of Proposition \ref{minor}.
}
\begin{prop}\label{minor}
	If $\mathcal{C}$ is a unimodular simplicial complex, then every minor of $\mathcal{C}$ is unimodular.
\end{prop}
\begin{proof}
	This follows immediately from Propositions \ref{induced} and \ref{link}.
\end{proof}

The fundamental example of a unimodular simplicial complex is the disjoint union of two simplices.
Here a simplex $\Delta_n$ is the simplicial complex on an $n+1$ element set with
a single facet consisting of all the elements.
We denote the irrelevant complex $\{\emptyset\}$ by $\Delta_{-1}$ and the void complex $\{\}$ by $\Delta_{-2}$.

\begin{prop}\label{disjointsimplices}
	Let $\mathcal{C} = \Delta_m \sqcup \Delta_n$ be the disjoint union of two simplices, for
	$m, n \geq 0$.
	Then $\mathcal{C}$ is unimodular.
\end{prop}

\begin{proof}
The matrix $\mathcal{A}_\calc$ in this case is the vertex-edge incidence matrix
of a complete bipartite graph with $2^{m+1}$ and $2^{n+1}$ vertices in the two
parts of the partition.  Such  vertex-edge incidence matrices are
examples of network matrices and are hence totally unimodular
\cite[Ch.~19]{Schrijver1986}.
\end{proof}

A second family of fundamental examples comes from taking the duals of
the disjoint union of two simplices.
We use $D_{m,n}$ to denote the dual of a disjoint union of an $m$-simplex 
and an $n$-simplex, i.e.
\[
	D_{m,n} := (\Delta_m \sqcup \Delta_n)^*.
\]
We close this section by giving a workable description of $D_{m,n}$.
We can divide the vertices of $D_{m,n}$ into disjoint sets $M,N$ such that
$M$ contains the vertices of the $\Delta_m$ in $D_{m,n}^*$ and $N$ contains the vertices of the $\Delta_n$ in $D_{m,n}^*$.
Then, the facets of $D_{m,n}$ are precisely the subsets of $M \sqcup N$ that leave out exactly one element of
$M$ and one element of $N$.
Notice that the complexes induced on $M$, $N$ are $\partial\Delta_m$ and $\partial\Delta_n$ respectively.
Also, note that for any $v \in M$, $\link_{v}(D_{m,n}) = D_{m-1,n}$ and for any $v \in N$, $\link_{v}(D_{m,n}) = D_{m,n-1}$.


\section{$\beta$-avoiding Simplicial Complexes}\label{sec:bavoid}

Part of our main result is a forbidden minor classification of unimodular simplicial complexes.
In this section, we identify these forbidden minors and prove various properties about the complexes that avoid them.

\begin{prop}\label{prop:badcomplexes}
The following simplicial complexes are minimal nonunimodular simplicial complexes:
	\begin{enumerate}
		\item $P_4$, the path on $4$ vertices
		\item $O_6$, the boundary of the octahedron or its dual $O_6^*$
		\item $J_1$, the complex on $\{1,2,3,4,5\}$ with facets $12,15,234,345$ or its 
		dual $J_1^*$, with facets $134, 235, 245$.
		\item $J_2$, complex on $\{1,2,3,4,5\}$ with facets $12,235,34,145$
		\item For $n \ge 1$, $\partial\Delta_n \sqcup \{v\}$, the disjoint union of the boundary of an $n$-simplex and a single vertex.
	\end{enumerate}
	
			\begin{figure}[h!]
			\scalebox{0.75}{\begin{tikzpicture}
				\vertex (a) at (0,0){};
				\vertex (b) at (1,0){};
				\vertex (c) at (2,0){};
				\vertex (d) at (3,0){};
				\path
					(a) edge (b)
					(b) edge (c)
					(c) edge (d)
				;
			\end{tikzpicture}}
			\\
			\scalebox{0.75}{\begin{tikzpicture}
				\vertex (a) at (0,0){};
				\vertex (b) at (-1,1){};
				\vertex (c) at (1,1){};
				\vertex (d) at (0,2){};
				\vertex (e) at (0,3){};
			\path
				(a) edge (b)
				(a) edge (c)
				(b) edge (c)
				(b) edge (d)
				(c) edge (d)
				(b) edge (e)
				(c) edge (e)
				(d) edge (e)
			;
				\draw[fill=gray] (a.center) -- (b.center) -- (c.center) -- cycle;
				\draw[fill=gray] (d.center) -- (b.center) -- (e.center) -- cycle;
				\draw[fill=gray] (d.center) -- (c.center) -- (e.center) -- cycle;
				\draw[fill=black] (e.center) circle(3pt);
				\draw[fill=black] (a.center) circle(3pt);
				\draw[fill=black] (b.center) circle(3pt);
				\draw[fill=black] (c.center) circle(3pt);
				\draw[fill=black] (d.center) circle(3pt);			
			\end{tikzpicture}}
			\qquad
			\scalebox{0.75}{\begin{tikzpicture}
				\vertex (a) at (0,0){};
				\vertex (b) at (1,1){};
				\vertex (c) at (2,0){};
				\vertex (d) at (1,-1){};
				\vertex (e) at (1,0){};
				\path
					(a) edge (b)
					(b) edge (c)
					(c) edge (d)
					(d) edge (a)
					(a) edge (e)
					(b) edge (e)
					(c) edge (e)
				;
				\draw[fill=gray] (a.center) -- (b.center) -- (e.center) -- cycle;
				\draw[fill=gray] (c.center) -- (b.center) -- (e.center) -- cycle;
				\draw[fill=black] (e.center) circle(3pt);
				\draw[fill=black] (a.center) circle(3pt);
				\draw[fill=black] (b.center) circle(3pt);
				\draw[fill=black] (c.center) circle(3pt);
			\end{tikzpicture}}
			\qquad
			\scalebox{0.75}{\begin{tikzpicture}
					\vertex (a) at (0,0){};
					\vertex (b) at (1,1){};
					\vertex (c) at (2,0){};
					\vertex (d) at (1,-1){};
					\vertex (e) at (1,0){};
					\path
						(a) edge (b)
						(b) edge (c)
						(c) edge (d)
						(d) edge (a)
						(a) edge (e)
						(b) edge (e)
						(c) edge (e)
						(d) edge (e)
					;
					\draw[fill=gray] (a.center) -- (b.center) -- (e.center) -- cycle;
					\draw[fill=gray] (c.center) -- (d.center) -- (e.center) -- cycle;
					\draw[fill=black] (e.center) circle(3pt);
					\draw[fill=black] (a.center) circle(3pt);
					\draw[fill=black] (b.center) circle(3pt);
					\draw[fill=black] (c.center) circle(3pt);
					\draw[fill=black] (d.center) circle(3pt);
				\end{tikzpicture}}
				\caption{The complexes $P_4, J_1, J_1^*,$ and $J_2$.}\label{fig:badcomplex}
		\end{figure}
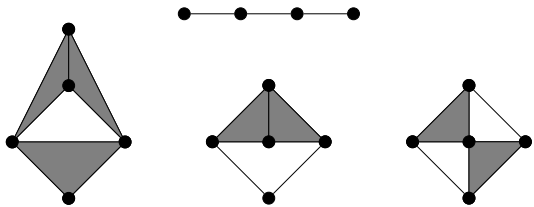
\end{prop}

Note that $P_4,J_2$ and $\partial\Delta_n \sqcup \{v\}$ are isomorphic 
to their own duals.

\begin{proof}
We can check that the complexes $P_4, O_6, O_6^*, J_1, J_1^*,$ and $J_2$ are not unimodular
by using the software 4ti2 \cite{4ti2} to compute the Graver basis and looking for entries that
are not $0, \pm 1$
In the case of $O_6$ and $O_6*$, these are too large to compute the
entire Graver basis.  However, selecting sufficiently large random subsets of the
columns produced Graver basis elements of the desired form.
{Computations are available on our website \cite{BernsteinWeb2016}.}
For the infinite family $\partial\Delta_n \sqcup \{v\}$ where $n \ge 1$,
examples of a non-squarefree Graver basis element  appear in \cite{bdypt}.
These results show that these examples are not unimodular.
To see that they are minimal, note that every subcomplex obtained by deleting
a single vertex or taking the link at a vertex produced a unimodular complex in
all cases.
\end{proof}

\begin{defn}
	A simplicial complex $\mathcal{C}$ is \emph{$\beta$-avoiding} if it does not contain any 
	of the complexes from Proposition \ref{prop:badcomplexes} as minors.
\end{defn}

Since none of the complexes from Proposition \ref{prop:badcomplexes} is
unimodular, Proposition \ref{minor} implies
that if $\mathcal{C}$ is unimodular, then $\mathcal{C}$ is $\beta$-avoiding.
The converse of this is our forbidden minor classification of unimodular complexes
which we prove in Theorem \ref{main}.

We now give some technical results about $\beta$-avoiding complexes.
Before beginning, we remind the reader that 
$\partial\Delta_1 \sqcup \{v\}$ is an independent set on $3$ vertices,
and so no $\beta$-avoiding complex can have an independent set of size $3$. 

\begin{prop}\label{bdual}
	Let $\mathcal{C}$ be a $\beta$-avoiding simplicial complex.
	Then $\mathcal{C}^*$ is also $\beta$-avoiding.
\end{prop}
\begin{proof}
	The list of prohibited minors of $\beta$-avoiding complexes is closed under taking duals.
	The proposition follows when we note that if $\mathcal{D}$ is a minor of $\mathcal{C}$,
	then $\mathcal{D}^*$ is a minor of $\mathcal{C}^*$.
	This is true because if $\mathcal{D}^*= \link_S(\mathcal{C}^*\setminus R)$,
	then two applications of Proposition \ref{linkdual} give
	\begin{align*}
		\mathcal{D}^* &= ((\mathcal{C}^*\setminus R)^*\setminus S)^*
		\\&= (\link_R(\mathcal{C})\setminus S)^*
	\end{align*}
	and therefore $\mathcal{D} = \link_R(\mathcal{C})\setminus S$.
	So $\mathcal{D}$ is a minor of $\mathcal{C}$.
\end{proof}

\begin{prop}\label{square}
	Let $\mathcal{C}$ be a $\beta$-avoiding simplicial complex that has $C_4$ induced.
	Then the complex induced on the non-ghost vertices of $\mathcal{C}$ is $\cone^p(C_4)$
	for some $p$.
\end{prop}

\begin{proof}
	Assume $\mathcal{C}$ has no ghost vertices.
	Let $u_1,\dots,u_4$ be the vertices that induce $C_4$ in $\mathcal{C}$.
	If these are the only vertices, we are done, so let $v$ be another vertex in $\mathcal{C}$.
	Let $\mathcal{D}$ denote the complex induced on $v,u_1,\dots,u_4$.
	Any minor of $\mathcal{D}$ is a minor of $\mathcal{C}$, so $\mathcal{D}$ must also be $\beta$-avoiding.
	Vertex $v$ cannot be disconnected from $u_1,\dots,u_4$ since otherwise $\mathcal{D}$ has an independent set of size $3$.
	Furthermore, $v$ must connect to $u_1,\dots,u_4$ with a triangle.
	Otherwise if $v$ connected to $u_i$ with an edge, then $u_i$ has edge degree 3
	and so $\link_v(\mathcal{D})$ is an independent set on $3$ vertices.
	There are $4$ possible triangles that $v$ can join with $u_1,\dots,u_4$ in $\mathcal{D}$
	so that $u_1,\dots,u_4$ induce $C_4$.
	If $v$ is only in one such triangle, wlog $\{v,u_1,u_2\}$,
	then $\{v,u_1,u_3,u_4\}$ is an induced $P_4$.
	The two non-isomorphic complexes on $5$ vertices that have two triangles and an induced $C_4$ are $J_1$ and $J_2$.
	If $\mathcal{D}$ has $3$ of the possible triangles, then $\link_v(\mathcal{D})$ is $P_4$.
	So $\mathcal{D}$ must have all four - i.e. it must be $\cone^1(C_4)$.

	Any vertices $v,v' \notin \{u_1, u_2, u_3, u_4\}$ must connect to each other.
	Otherwise, the induced complex on $v,v',u_1,\dots,u_4$ is $O_6$ which is 
	not unimodular.
	Now we denote the vertices of $\mathcal{C}$ that are not any of the $u_i$s as $v_1,\dots,v_k$.
	Since $v_k$ is a cone point over the $C_4$ on $u_1,\dots,u_4$ and since $v_k$ connects to all $v_j, j<k$,
	$\link_{v_k}(\calc)$ is a $\beta$-avoiding complex on $v_1,\dots,v_{k-1},u_1,\dots,u_4$ that has $C_4$ induced on $u_1,\dots,u_4$.
	So, by induction on the number of vertices, $\link_{v_k}(\calc)$ is an iterated cone over $C_4$, $\cone^{k-1}(C_4)$.
	But also the induced complex on $v_1,\dots,v_{k-1},u_1,\dots,u_4$ is 
	$\cone^{k-1}(C_4)$ by induction.  This implies that $\calc = \cone^k(C_4)$.
\end{proof}

\begin{prop}\label{iteratedcone}
	Let $\mathcal{C}$ be a $\beta$-avoiding complex that has $D_{m,n}$ induced for some $m,n \ge 1$.
	Then the complex induced on the non-ghost vertices of $\mathcal{C}$ is $\cone^p(D_{m,n})$
	for some $p$.
\end{prop}
\begin{proof}
	Assume $\mathcal{C}$ has no ghost vertices.
	Just as in the remarks at the end of section \ref{sec:basic},
	we divide the vertices of $D_{m,n}$ into disjoint sets $M,N$ such that
	$M$ contains the vertices of the $\Delta_m$ in $D_{m,n}^*$ and $N$ contains the vertices of the $\Delta_n$ in $D_{m,n}^*$.
	Now, we proceed by induction on $m$ and $n$.
	Note that $D_{1,1} = C_4$, so the base case is handled by Proposition \ref{square}.
	Assume that, without loss of generality, $m \geq 2$.
	
	Let $v \in \mathcal{C} \setminus D_{m,n}$.
	The vertex $v$ must connect to some $u \in M$ to avoid inducing $\partial\Delta_m \sqcup \{v\}$.
	Then, $\link_u(\calc)$ contains $v$ and $D_{m-1,n}$, and so by induction,
	the complex induced on $v$ and the vertices of $D_{m-1,n}$ in $\link_u(\calc)$ is a cone over $D_{m-1,n}$.
	This means that $v$ is in every facet that contains $u$.
	Since $m \geq 2$,
	$u$ is connected to every other vertex in $D_{m,n}$ and thus
	 $v$ is attached to every vertex in $D_{m,n}$.
	Now we apply the same argument to each of the vertices in the set $M$
	to see that $v$ is in every facet that contains any vertex in set $M$.
	Every facet of $D_{m,n}$ contains some element of $M$.
	So this implies that the induced complex on $v$ and
	$D_{m,n}$ must be the cone over $D_{m,n}$.

	Now assume $v,v'$ are both vertices in $\mathcal{C} \setminus D_{m,n}$.
	If they were not connected by an edge, then the induced complex on $v, v'$ and $D_{m,n}$
	contains a minor which is isomorphic to $O_6$.  This means that
	$v$, $v'$ are connected by an edge.  Taking the link $\link_v(\calc)$ produces
	a smaller complex with an induced $D_{m,n}$ hence it must be a cone
	by induction on the number of vertices not in the $D_{m,n}$.
	This reduces us the case where $\link_v(\calc)$ is $\cone^{p-1}(D_{m,n})$
	and $\calc \setminus v$ is $\cone^{p-1}(D_{m,n})$ which implies that
	$\calc$ is $\cone^p(D_{m,n})$. 
\end{proof}

\begin{prop}\label{ksimplex}
	Let $\mathcal{C}$ be a $\beta$-avoiding simplicial complex on vertices $u_1,\dots,u_{k+1},v$ such
	that the complex induced on $\{u_1,\dots,u_{k+1}\}$ is $\partial\Delta_{k}$.
	Then $v$ must be in a $k$-simplex with some subset of the $u_i$'s.
\end{prop}

\begin{proof}
	We proceed by induction on $k$.
	For the base case $k=1$, note that $\partial\Delta_1$ is two isolated points.
	In this case, $v$ must connect to $u_1$ or $u_2$
	as a $1$-simplex (an edge) to avoid inducing $\partial\Delta_1 \sqcup \{v\}$.
	\\
	\indent
	Now assume $k>1$.
	The vertex $v$ must attach to some $u_i$ to avoid inducing $\partial\Delta_k \sqcup \{v\}$.
	Then, $\link(u_i)$ has $v$ and $\partial\Delta_{k-1}$,
	so by induction, $v$ must form a $k-1$ simplex with some collection $u_1,\dots,\hat{u_i},\dots,\hat{u_j},\dots,u_{k+1}$.
	So in $\mathcal{C}$, $v$ must be in a $k$ simplex with $u_1,\dots,\hat{u_j},\dots,u_{k+1}$.
\end{proof}


\begin{prop}\label{minimality}
	Let $\mathcal{C}$ be a $\beta$-avoiding simplicial complex on $m+n+2$ vertices that does not have a facet of dimension $m+n$.
	Assume $D_{m,n} \subseteq \mathcal{C}$.
	Then $\mathcal{C} = D_{m,n}$.
\end{prop}
\begin{proof}
	{ Note that $\mathcal{C}^* \subseteq \Delta_m \sqcup \Delta_n$ because $D_{m,n} \subseteq \mathcal{C}$.
	Since $\mathcal{C}$ does not have a facet of dimension $m+n$, $\calc^*$ has no ghost vertices.
	Since $\calc^*$ has no ghost vertices, if $\calc^* \subsetneq \Delta_m \sqcup \Delta_n$,
	then $\calc^*$ would have an induced $\{v\} \sqcup \partial \Delta_k$ for some $k \geq 1$.
	In this case $\calc^*$ would not be $\beta$-avoiding and so by Proposition \ref{bdual}, neither would $\calc$.}
\end{proof}


\section{The $1$-Skeleton of a $\beta$-avoiding Complex}\label{sec:1skel}

In this section we prove Lemma \ref{1skel} which gives a complete characterization of the $1$-skeleton of a $\beta$-avoiding simplicial complex.
This is a crucial technical lemma in the proof of Theorem \ref{main}.
We start with a technical proposition about graphs.

\begin{prop}\label{bipartite}
	Let $H$ be a connected graph that avoids $K_3$ and $P_4$ as induced subgraphs.
	Then $H$ is a complete bipartite graph.
\end{prop}
\begin{proof}
	Let $u \in V(H)$, let $N(u)$ denote the neighbors of $u$,
	and let $M(u)$ denote the non-neighbors of $u$ (this set includes $u$).
	The bipartition of the vertices of $H$ will be $M(u)$ and $N(u)$.
	Let $v \in M(u)\setminus\{u\}$.
	Since $H$ is connected, there exists a path $u=u_1,u_2,\dots,u_k=v$.
	Assume $k$ is minimal.
	Since $v \neq u$, $k > 1$.
	We cannot have $k=2$ since $u,v$ are non-neighbors.
	We cannot have $k = 4$, since in order to avoid an induced $P_4$, we would need an edge $(u_i,u_{i+2})$ inducing a $K_3$,
	or an edge $(u,v)$ contradicting that $u,v$ are non-neighbors.
	If $k\ge 5$, there must exist an edge $(u_1,u_4)$ to avoid an induced $P_4$, contradicting minimality of $k$.
	So we have $k=3$.
	So for any $v\in M(u)\setminus\{u\}$, there exists a path $u,a,v$.
	\\
	\indent
	Now we show that $H$ is bipartite with bipartition $M(u)$ and $N(u)$.
	It is clear that $N(u)$ is an independent set of vertices, for if $v,w \in N(u)$ had an edge between them,
	there would be a $K_3$ on $u,v,w$.
	Now we show that $M(u)$ is an independent set of vertices.
	Assume $w,v\in M(u)$.
	If either $w,v$ is $u$, there is no edge between them by definition of $M(u)$, so assume $w,v \neq u$.
	Then by the above, we have paths $u,a,v$ and $u,b,w$.
	{ So an edge $(v,w)$ would induce the $5$-cycle $u,a,v,w,b,u$.
	This $5$-cycle must have a chord to avoid inducing $P_4$, but any chord in a $5$-cycle induces a $K_3$.}
	\\
	\indent
	Now we show that $H$ is complete bipartite.
	Let $x \in M(u)$ and $y \in N(u)$.
	Since $H$ is connected, there is a path $x=u_1,...,u_k=y$.
	We may without loss of generality assume $k\le 3$ since otherwise we could shorten the path using the edge $(u_1,u_4)$ required to avoid a $P_4$.
	The sets $M(u)$ and $N(u)$ are disjoint, so $k \neq 1$.
	If $k=3$, then there is an induced $P_4$ $x,u_2,y,u$.
	So $k=2$, and so $(x,y)$ is an edge.
\end{proof}

For a graph $G$ we define the complement graph $G^c$ on the same set of vertices
such that $(u,v)$ is an edge of $G^c$ if and only if $(u,v)$ is not an edge in $G$.
Now we can use Proposition \ref{bipartite} to give a strong restriction on the structure of $G^c$ whenever $G$
is the $1$-skeleton of a $\beta$-avoiding complex.

\begin{prop}\label{complement}
	If $G$ is the $1$-skeleton of a $\beta$-avoiding simplicial complex $\mathcal{C}$,
	then each connected component of $G^c$ is complete bipartite.
\end{prop}
\begin{proof}
	Note that $\{v\} \sqcup \partial \Delta_1$ consists of three disconnected vertices, its complement graph is $K_3$.  The path $P_4$ is its own complement.
	Since $\mathcal{C}$ is $\beta$-avoiding,
	$G$ avoids $P_4$ and an independent set of size three as induced subgraphs.
	So each connected component of $G^c$ avoids $P_4$ and $K_3$ as induced subgraphs.
	Proposition \ref{bipartite} therefore implies that each connected component of $G^c$ is complete bipartite.
\end{proof}

Now we are ready to characterize the $1$-skeleton of a $\beta$-avoiding simplicial complex.

\begin{lemma}\label{1skel}
	Let $\mathcal{C}$ be a $\beta$-avoiding simplicial complex and let $G$ denote its $1$-skeleton.
	Then $G$ is one of the following
	\begin{itemize}
		\item[(a)] $K_N$
		\item[(b)] Two complete graphs glued along a (possibly empty) common clique
		\item[(c)] The iterated cone over a $4$-cycle.
	\end{itemize}
\end{lemma}
\begin{proof}
	By Proposition \ref{complement}, we know that each connected component of $G^c$ is complete bipartite.
	Let $H$ denote the induced subgraph of $G^c$ that removes all isolated vertices.
	If $H$ is empty, then $G^c$ is an independent set of vertices and therefore $G = K_N$.
	So assume $H$ is nonempty.
	We claim that if $H$ is neither $K_{m,n}$ nor $K_2 \sqcup K_2$ for $m,n \ge 1$,
	then $H$ induces either $K_2 \sqcup K_2 \sqcup K_2$ or $P_3 \sqcup K_2$.
	To prove the claim, first assume that $H$ avoids $K_2 \sqcup K_2 \sqcup K_2$.
	Since $H$ is not $K_{m,n}$ and has no isolated vertices, this implies that $H$ has exactly two components.
	Since $H$ is not $K_2 \sqcup K_2$, some connected component of $H$ has at least three vertices.
	So we have a $P_3$ induced by this component, and $K_2$ induced by the other component and so the claim is proven.
	\\
	\indent
	The complement graphs of $K_2 \sqcup K_2 \sqcup K_2$ or $P_3 \sqcup K_2$
	are shown in Figure \ref{fig:complements}.
	Both graphs have $C_4$ induced, but neither is the $1$-skeleton
	for an iterated cone over $C_4$ since neither graph has a suspension vertex.
	Proposition \ref{square} therefore implies that neither is the $1$-skeleton of a $\beta$-avoiding simplicial complex
	and so $G$ may not induce either.
	So $G^c$ may not induce $K_2 \sqcup K_2 \sqcup K_2$ nor $P_3 \sqcup K_2$ and therefore neither may $H$.
	The claim then implies that $H$ must be either $K_{m,n}$ or $K_2 \sqcup K_2$.
	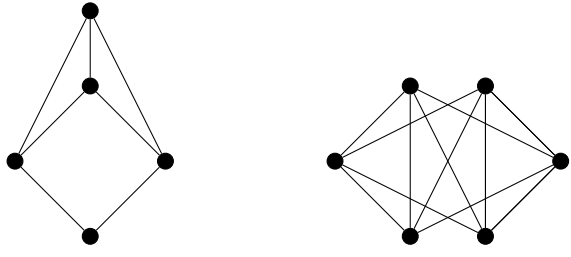
\begin{figure}
		\begin{tikzpicture}
			\vertex (a) at (1,0){};  
			\vertex (b) at (2,1){};
			\vertex (c) at (0,1){};
			\vertex (d) at (1,2){};				
			\vertex (e) at (1,3){};
			\path
				(a) edge (b)
				(b) edge (d)
				(d) edge (c)
				(c) edge (a)
				(e) edge (b)
				(e) edge (c)
				(e) edge (d)
			;   
		\end{tikzpicture}
		\hspace{50pt}
		\begin{tikzpicture}
			\vertex (a) at (2,0){};  
			\vertex (b) at (3,1){};
			\vertex (c) at (2,2){};
			\vertex (a') at (1,0){};
			\vertex (b') at (0,1){};
			\vertex (c') at (1,2){};
			\path
				(a) edge (b)
				(a) edge (c)
				(a) edge (b')
				(a) edge (c')
				(b) edge (a)
				(b) edge (c)
				(b) edge (a')
				(b) edge (c')
				(c) edge (b)
				(c) edge (a)
				(c) edge (b')
				(c) edge (a')
				(a') edge (b')
				(b') edge (c')
				(c') edge (a')
			;   
		\end{tikzpicture}
		\caption{The complement graphs of $P_3 \sqcup K_2$ and $K_2 \sqcup K_2 \sqcup K_2$}\label{fig:complements}
	\end{figure}
	Assume $G^c$ has $p$ isolated vertices.
	If $H = K_{m,n}$ then $G$ is a $K_{m+p}$ and a $K_{n+p}$ glued along a common $K_p$.
	If $H= K_2 \sqcup K_2$ then $G$ is an iterated cone over a $4$-cycle.
\end{proof}

As we see in the following proposition, the $1$-skeleton of $\calc$ completely determines $\calc$ when it is
obtained by gluing two complete graphs along an empty clique.

\begin{prop}\label{disconnected}
	Let $\mathcal{C}$ be a $\beta$-avoiding simplicial complex such that the $1$-skeleton
	of $\mathcal{C}$ is the disjoint union of cliques $K_m$ and $K_n$.
	Then $\mathcal{C} = \Delta_m \sqcup \Delta_n$.
\end{prop}
\begin{proof}
	Let $v_1,\dots,v_k$ be vertices of the $K_m$.
	If $\{v_1,\dots,v_k\}$ is a minimal non-face of $\mathcal{C}$, then if we let $u$
	be a vertex in the $K_n$, then the complex indued on $v_1,\dots,v_k,u$ is $\partial\Delta_k \sqcup \{u\}$.
\end{proof}


\section{The Main Theorem}\label{sec:main}

The goal of this section is to give a proof of Theorem \ref{main}
which gives a complete characterization of the unimodular binary
hierarchical models.

We begin this section by defining a \emph{nuclear} complex.
A nuclear complex is a complex that can be obtained from a disjoint union
of two simplices by adding cone vertices, adding ghost vertices, taking Lawrence liftings
and taking Alexander duals.
Since $\Delta_m \sqcup \Delta_n$ is unimodular and these operations all preserve unimodularity,
nuclear complexes are unimodular.
Part of our main result is the converse - unimodular complexes are nuclear.

\begin{defn}\label{defn:nuclear}
	A simplicial complex $\mathcal{C}$ is \emph{nuclear} if one of the following is true
	\begin{enumerate}
		\item\label{nuclear:lawrence} $\mathcal{C} = \Lambda(\mathcal{D})$ where $\mathcal{D}$ is nuclear
		\item\label{nuclear:ghost} $\mathcal{C} = G(\mathcal{D})$ where $\mathcal{D}$ is nuclear
		\item\label{nuclear:disjoint} $\mathcal{C} = \cone^p(\Delta_m \sqcup \Delta_n)$ for $p,m,n \ge 0$
		\item\label{nuclear:dual} $\mathcal{C} = \cone^p(D_{m,n})$ for $p \ge 0$ and $m,n \ge 1$
		\item\label{nuclear:simplex} $\mathcal{C} = \Delta_k$ for $k \ge -2$.
	\end{enumerate}
	Every nuclear complex $\mathcal{C}$ can be constructed by applying the operations
	$\cone^p(\cdot)$, $G(\cdot)$, and $\Lambda(\cdot)$ to a complex $\mathcal{D}$
	where $\mathcal{D}$ is of the form $\Delta_m \sqcup \Delta_n$, $D_{m,n}$, or $\Delta_k$.
	We refer to $\mathcal{D}$ as the \emph{nucleus} of $\mathcal{C}$.
\end{defn}
Note that $D_{m,0}$ has a ghost vertex for all $m$.
This is why we have $m,n \ge 1$ in (\ref{nuclear:dual}).
We note that the collection of nuclear complexes is closed under Alexander duality.

\begin{prop}\label{dualclosed}
	If $\mathcal{C}$ is nuclear then so is $\mathcal{C}^*$.
\end{prop}
\begin{proof}
	Possibilities (\ref{nuclear:lawrence}) and (\ref{nuclear:ghost}) are dual to each other,
	as are (\ref{nuclear:disjoint}) and (\ref{nuclear:dual}).
	On $k$ vertices, $\Delta_k$ and $\Delta_{-2}$ are dual to each other.
\end{proof}

We now state our main result.

\begin{thm}\label{main}
	Let $\mathcal{C}$ be a simplicial complex.
	Then the following are equivalent
	\begin{enumerate}
		\item\label{unimodular} $\mathcal{C}$ is unimodular
		\item\label{bavoiding} $\mathcal{C}$ is $\beta$-avoiding
		\item\label{nuclear} $\mathcal{C}$ is nuclear.
	\end{enumerate}
\end{thm}

{ We defer the proof of Theorem \ref{main} until the end of the section,
but we give a roadmap here.
It is immediate from results in previous sections that nuclear complexes are unimodular
and that unimodular complexes are $\beta$-avoiding.
\\
\indent
We show that any $\beta$-avoiding complex $\calc$ is nuclear by 
choosing a particular vertex $v$ of $\calc$ and using induction on the number of vertices to conclude that $\link_v(\calc)$ is nuclear.
From here, we have five cases to consider - one for each of the ways that $\link_v(\calc)$ can be nuclear.
In four of these cases, it is relatively easy to show that $\calc$ is nuclear.
The difficulty lies in the case where $\link_v(\calc) = D_{m,n}$, which is handled in Lemma \ref{case4}.
The proof of Lemma \ref{case4} is split into two main parts
and the second part is further split into six cases.
Proposition \ref{messy1} helps with the first part.
Propositions \ref{m'n'}, \ref{n}, and \ref{arbitrary} handle the hard cases in the second part.
}

\begin{prop}\label{messy1}
	Let $\mathcal{C}$ be a $\beta$-avoiding simplicial complex.
	Assume that all proper minors of $\mathcal{C}$ are nuclear.
	Assume $\mathcal{C}$ has a vertex $v$ such that $\mathcal{C}\setminus v = \cone^p(\Delta_m \sqcup \Delta_n)$
	with $m,n \ge 1$ and $p \ge 1$.
	Assume $\mathcal{C}$ is $C_4$-free.
	Then $\mathcal{C} = \cone^{p+1}(\Delta_m \sqcup \Delta_n)$,  $\mathcal{C} = \cone^p(\Delta_{m+1}\sqcup \Delta_n)$, or $\mathcal{C} = \cone^p(\Delta_{m}\sqcup \Delta_{n+1})$.
\end{prop}

\begin{proof}
	We induct on $p$.
	For the base case, take $p = 1$.
	If we let $u$ denote the cone vertex in $\mathcal{C}\setminus v$
	then $v$ must connect to one of the simplices in the $\Delta_m\sqcup \Delta_n \subset \mathcal{C}\setminus u$.
	Otherwise we have an independent set with three vertices induced in $\mathcal{C}\setminus u$ and this complex is not nuclear.
	If $v$ connects to only one such simplex, then $v$ must also connect to $u$ to avoid inducing $P_4$ which is not nuclear.
	If $v$ connects to both such simplices, and $v$ does not connect to $u$, then $\mathcal{C}$
	has $C_4$ induced which contradicts our hypothesis.
	{ So assume $v$ connects to $u$ and to at least one simplex.}
	Since $u$ is a cone vertex in $\mathcal{C}\setminus v$, $\link_u(\mathcal{C})$ has no ghost vertices and by hypothesis, it is nuclear.
	Furthermore, $\link_u(\mathcal{C}\setminus v) = \Delta_{m}\sqcup \Delta_n$
	and so $\link_u(\mathcal{C})$ is either $\Delta_{m+1}\sqcup \Delta_n$
	or $\cone^1(\Delta_{m}\sqcup\Delta_{n})$
	(note that these are the only nuclear complexes that become disconnected upon removing a vertex).
	In the first case, $\mathcal{C}= \cone^1(\Delta_{m+1}\sqcup \Delta_n)$
	and in the second case $\mathcal{C}=\cone^2(\Delta_{m}\sqcup\Delta_{n})$.
	\\
	\indent
	Now assume $p \ge 2$.
	Let $u_1,\dots,u_p$ denote the cone vertices of $\mathcal{C}\setminus v$.
	Let $M,N$ denote the vertex sets of the $\Delta_m,\Delta_n$ respectively.
	Then $v$ must connect to at least one of $M$ and $N$ to avoid inducing an independent set on three vertices.
	Furthermore, $v$ must connect to each $u_i$ -
	if $v$ only connects to one of $M$ or $N$, then this is required to avoid $P_4$,
	and if $v$ connects to both then
	we need this to ensure that $\mathcal{C}$ is $C_4$-free.
	Furthermore, the set $\{u_1,\dots,u_p,v\}$ is a facet of $\mathcal{C}$.
	This is clear when $p = 1$ since otherwise $v$ doesn't connect to $u_1$.
	Induction on $p$ gives that $\{u_1,\dots,u_{p-1},v\}$ is a facet of $\link_{u_p}(\mathcal{C})$
	and therefore $\{u_1,\dots,u_p,v\}$ is  a facet of $\mathcal{C}$.
	\\
	\indent
	Now, we can see that $\link_{u_p}(\mathcal{C})$ has no ghost vertices and is $C_4$-free.
	Furthermore, $\link_{u_p}(\mathcal{C})\setminus v$ and $\calc \setminus \{u_p, v\}$
	are both equal to $\cone^{p-1}(\Delta_m \sqcup \Delta_n)$.
	So by induction on $p$, each of $\link_{u_p}(\mathcal{C})$ and $(\mathcal{C}\setminus u_p)$
	can either be $\cone^p(\Delta_m \sqcup \Delta_n)$ or $\cone^{p-1}(\Delta_{m+1}\sqcup \Delta_n)$.
	If $\link_{u_p}(\mathcal{C}) = \mathcal{C}\setminus u_p$, then $u_p$ is a cone vertex in $\mathcal{C}$.
	So if both equal $\cone^p(\Delta_m \sqcup \Delta_n)$ then
	$\mathcal{C} = \cone^{p+1}(\Delta_m \sqcup \Delta_n)$.
	If both equal $\cone^{p-1}(\Delta_{m+1}\sqcup \Delta_n)$
	then $\mathcal{C} = \cone^{p}(\Delta_{m+1}\sqcup \Delta_n)$.
	\\
	\indent
	Since $\link_{u_p}(\mathcal{C}) \subseteq \mathcal{C}\setminus u_p$,
	it is impossible to have $\link_{u_p}(\mathcal{C}) = \cone^p(\Delta_m \sqcup \Delta_n)$
	and $\mathcal{C}\setminus u_p = \cone^{p-1}(\Delta_{m+1} \sqcup \Delta_n)$.
	It is also impossible to have $\mathcal{C}\setminus u_p = \cone^p(\Delta_m \sqcup \Delta_n)$
	and $\link_{u_p}(\mathcal{C}) = \cone^{p-1}(\Delta_{m+1} \sqcup \Delta_n)$.
	For if this were the case,
	then $v$ would be a cone vertex in $\mathcal{C}\setminus u_p$
	and it would be part of the $\Delta_{m+1}$ in $\link_{u_p}(\mathcal{C})$.
	Since $n \ge 1$, $\# N \ge 2$, so given $b_1,b_2 \in N$ and $a \in M$ the complex induced
	on $a,b_1,b_2,u_p,v$ has facets $\{v,a,u_p\},\{v,b_1,b_2\},\{u_p,b_1,b_2\}$.
	This is $J_1^*$, and so $\mathcal{C}$ is not $\beta$-avoiding.
\end{proof}

{
\begin{prop}\label{m'n'}
	Let $\mathcal{C}$ be the simplicial complex on $q$ vertices
	with a vertex $v$ such that $\link_v(\mathcal{C}) = D_{m,n}$ with $m,n \ge 1$
	and $(\mathcal{C}\setminus v)^* = G^{m+n-m'-n'}(\Delta_{m'} \sqcup \Delta_{n})$ with $-1 \le m' < m$ and $0 \le n' \le n$.
	Unless $m' = -1$ and $n' = n$, $\mathcal{C}$ has a proper minor that is not nuclear.
\end{prop}
\begin{proof}
	Let $M,N$ denote the sets of vertices such that
	in $(\link_v(\mathcal{C}))^*$, the complex induced on $M$ is $\Delta_m$
	and the complex induced on $N$ is $\Delta_n$.
	Then some vertex $x \in N$ is a non-ghost in $(\mathcal{C}\setminus v)^*$.
	We show that $\link_x(\mathcal{C})$ is not nuclear.
	\\
	\indent
	First we describe the facets of $\link_x(\mathcal{C})$.
	We claim that the facets of $\link_x(\mathcal{C})$
	are precisely the sets containing $N-1-2$ vertices,
	where the omitted vertex pair is one of the following
	\begin{enumerate}
		\item some $a \in M$ and some $b \in N\setminus\{x\}$ or
		\item $v$ and some vertex $g$ that is a ghost in $(\mathcal{C}\setminus v)^*$.
	\end{enumerate}
	We now prove the claim.
	The facets of $\link_x(\mathcal{C})$ that contain $v$
	are the sets whose complements are pairs of the first type since $\link_v(\mathcal{C}) = D_{m,n}$.
	The facets of $\link_x(\mathcal{C})$ that do not contain $v$
	are the sets whose complements are minimal non-faces of $(\mathcal{C}\setminus v)^*$
	that do not become faces of $\link_x(\mathcal{C})$ upon adding $v$.
	These minimal non-faces of $(\mathcal{C}\setminus v)^*$ are the ghost vertices.
	Therefore the facets of $\link_x(\calc)$ that do not contain $v$ are the sets of the second type.
	\\
	\indent
	We now show that $\link_x(\mathcal{C})$ is not nuclear.
	The complex $\link_x(\mathcal{C})$ cannot be $\Lambda(\mathcal{D})$
	for any complex $\mathcal{D}$ since no facet lacks fewer than two vertices.
	Nor can $\link_x(\mathcal{C})$ have any ghost vertices -
	$m \ge 1$ implies $\#M \ge 2$ and so every vertex in $M$ is in a facet of the first type,
	every vertex of $N$ is in a facet of the second type, and $v$ is in every facet of the first type.
	It is clear that no vertex is in every facet,
	so $\link_x(\mathcal{C})$ cannot be $\cone^p(D_{k,l})$ nor $\cone^p(\Delta_k \sqcup \Delta_l)$
	for any $p \ge 1$.
	It is clear that $\link_x(\mathcal{C})$ has more than 2 facets, so $\link_x(\mathcal{C})$ cannot be $\Delta_k \sqcup \Delta_l$,
	nor $\Delta_k$.
	\\
	\indent
	The only case left to check is $\link_x(\mathcal{C}) = D_{k,l}$.
	If this is the case then we can partition the vertices of $\link_x(\mathcal{C})$ into disjoint sets $C_1,C_2$ such that
	the facets of $\link_x(\mathcal{C})$ are precisely the subsets of $C_1 \cup C_2$ that omit exactly one from each $C_1$ and $C_2$.
	\\
	\indent
	Assume $n' = n$ and $m' \ge 0$.
	Let $g\in M$ be a ghost vertex and without loss of generality assume $g \in C_1$.
	Facets of the first type imply $M \subseteq C_1$.
	Facets of the second type imply $v \in C_2$.
	Since $m' \ge 0$, there exists some $a \in M$ that is not a ghost in $(\calc\setminus v)^*$.
	Snd so $(C_1\cup C_2)\setminus\{a,v\}$ is a facet of $\link_x(\calc)$.
	But this does not fit our earlier description of the facets of $\link_x(\calc)$.
	\\
	\indent
	Now assume $-1\le m'<m$ and $0 \le n'<n$.
	Then there exist $g' \in M$ and $g'' \in N$ that are ghosts in $(\calc\setminus v)^*$.
	But this contradicts our earlier description of the facets of $\link_x(\calc)$ -
	facets of the first kind imply that $g',g''$ are in different $C_i$s
	and facets of the second kind imply that they are in the same $C_i$.
\end{proof}
}

\begin{prop}\label{n}
	Let $\mathcal{C}$ be a simplicial complex on $q$ vertices.
	Let $v$ be a vertex in $\mathcal{C}$ such that $\link_v(\mathcal{C}) = D_{m,n}$ with $m,n \ge 1$.
	Assume $(\mathcal{C}\setminus v)^* = G^{m+1}(\Delta_{n})$.
	Then $\mathcal{C} = D_{m,n+1}$.
\end{prop}
\begin{proof}
	Let $M,N$ denote the sets of vertices such that
	in $(\link_v(\mathcal{C}))^*$, the complex induced on $M$ is $\Delta_m$
	and the complex induced on $N$ is $\Delta_n$.
	Then all the vertices of $M$ and none of the vertices of $N$ are ghosts in $(\mathcal{C}\setminus v)^*$.
	We claim that each facet of $\mathcal{C}$ contains $N-2$ vertices, and the omitted pair is one of the following
	\begin{enumerate}
		\item some $a \in M$ and some $b \in N$
		\item $v$ and some $a \in M$.
	\end{enumerate}
	To see this, note that if a facet contains $v$, then it must be of the first form since $\link_v(\mathcal{C}) = D_{m,n}$.
	If a facet does not contain $v$, then its complement must be $v$, along with a minimal non-face of $(\mathcal{C}\setminus v)^*$.
	These minimal non faces are precisely the vertices in $M$.
	So $\mathcal{C} = D_{m,n+1}$ where $M$ and $N \cup \{v\}$
	are the sets of vertices for the $\Delta_{m}$ and $\Delta_{n+1}$ in $\mathcal{C}^*$ respectively.
\end{proof}

\begin{prop}\label{arbitrary}
	Let $\mathcal{C}$ be a simplicial complex on $q$ vertices.
	Let $v$ be a vertex in $\mathcal{C}$ such that $\link_v(\mathcal{C}) = D_{m,n}$ with $m,n \ge 1$
	and has no ghost vertices.
	Let $M,N$ denote the sets of vertices such that
	in $(\link_v(\mathcal{C}))^*$, the complex induced on $M$ is $\Delta_m$
	and the complex induced on $N$ is $\Delta_n$.
	Assume $(\mathcal{C}\setminus v)^* = G^{m+1}(\mathcal{D})$ where $\mathcal{D}$ is a nuclear complex on vertex set $N$
	with more than one facet.
	Then $\mathcal{C}$ has a proper minor that is not nuclear.
\end{prop}
\begin{proof}
	Let $x \in M$.
	We show that $\link_x(\mathcal{C})$ is not nuclear.
	{
	We start by showing it is not Lawrence type.
	Let $F$ be a facet of $\link_x(\mathcal{C})$ that does not contain $v$.
	Then $F \cup \{x\}$ is a facet of $\calc\setminus v$ so it must also lack some $a \in M \setminus \{x\}$.
	If $F$ is a facet of $\link_x(\calc)$ that contains $v$, then $F \cup \{x\} \setminus \{v\}$ is a facet of $\link_v(\calc)$
	and therefore lacks some $a \in M$ and $b\in N$.
	So $\link_x(\mathcal{C})$ is not $\Lambda(\mathcal{F})$ for any complex $\mathcal{F}$.
	\\
	\indent
	We can see that $\link_x(\mathcal{C})$ has more than two facets - since $m,n \ge 1$, there are at least two that include $v$
	and since $(C\setminus v)^* = G^{m+1}(\mathcal{D})$, there is at least one that does not.}
	$\link_x(\mathcal{C})$ cannot be $\cone^p(\Delta_k \sqcup \Delta_l)$ for any $p \ge 0$.
	\\
	\indent
	All of the vertices of $M$ are ghost vertices in $(\mathcal{C}\setminus v)^*$,
	and $(\mathcal{C}\setminus v)^*$ has a minimal non-face $S \subseteq N$ with at least two vertices.
	This implies that $\mathcal{C}$ contains a facet $F$ such that $M \subseteq F$,
	there exist $b_1,b_2 \in N$ such that $b_1,b_2 \notin F$ and $v \notin F$ (since $\link_v(\mathcal{C}) = D_{m,n}$).
	Facets in $\link_x(\mathcal{C})$ that contain $v$ have $N-1-2$ vertices - they lack some $a \in M\setminus\{x\}$ and some $b \in N$.
	Furthermore, $F\setminus\{x\}$ is a facet of $\link_x(\mathcal{C})$ containing at most $N-1-3$ vertices.
	This means that $\link_x(\mathcal{C})$ is not pure, and therefore not $\cone^p(D_{k,l})$ for any $p \ge 0$.
	So $\link_x(\mathcal{C})$ is a proper minor of $\mathcal{C}$ that is not nuclear.
\end{proof}

\begin{lemma}\label{case4}
	Let $\mathcal{C}$ be a $\beta$-avoiding simplicial complex on $q$ vertices.
	Assume all proper minors of $\mathcal{C}$ are nuclear.
	If there exists a vertex $v$ of $\mathcal{C}$ such that $\link_v(\mathcal{C}) = \cone^p(D_{m,n})$
	for $p \ge 0$ and $m,n \ge 1$,
	then $\mathcal{C}$ is nuclear.
\end{lemma}
\begin{proof}
	{We split this into two cases.
	For the first case, assume $p \ge 1$.
	Here, Propositions \ref{linkdual} and \ref{commutes} give $\mathcal{C}^*\setminus v = \cone^p(\Delta_m\sqcup\Delta_n)$.
	By Proposition \ref{bdual},
	$\calc^*$ is $\beta$-avoiding so in $\calc^*$, $v$ must connect to one of the simplices in the $\Delta_m\sqcup \Delta_n \subset \mathcal{C}^*$
	to avoid inducing an independent set on three vertices.
	Now let $u$ be a cone vertex.
	If $v$ connects to only one such simplex, then $v$ must also connect to $u$ to avoid inducing $P_4$.
	If $v$ connects to both such simplices, and $v$ does not connect to $u$, then $\mathcal{C}^*$
	has $C_4$ induced.
	In this case Proposition \ref{square} implies $\mathcal{C}^* = \cone^k(C_4)$
	and so $\mathcal{C} = \cone^k(\Delta_1 \sqcup \Delta_1)$.
	So we can assume $v$ connects to at least one simplex, and to all cone vertices in $\calc^*\setminus v$.
	From this we can see that $\mathcal{C}^*$ has no induced $C_4$.
	Since every proper minor of $\mathcal{C}^*$ is nuclear, Proposition \ref{messy1} implies
	$\mathcal{C}^* = \cone^{p+1}(\Delta_m \sqcup \Delta_n)$ or $\mathcal{C}^* = \cone^p(\Delta_{m+1}\sqcup \Delta_n)$.
	Proposition \ref{dualclosed} then implies that $\calc$ is nuclear.}
	\\
	\indent
	For the second case, assume $p = 0$.
	So $\link_v(\mathcal{C}) = D_{m,n}$ with $m,n \ge 1$.
	Then $D_{m,n} \subseteq \mathcal{C}\setminus v$
	and so $(\mathcal{C}\setminus v)^* \subseteq \Delta_m \sqcup \Delta_n$.
	Let $M,N$ denote the sets of vertices of $\mathcal{C}$ such that 
	in $(\link_v(\mathcal{C}))^*$, the vertices of the $\Delta_m$ are $M$, and the vertices of the $\Delta_n$ are $N$.
	The non-ghost vertices of $(\mathcal{C}\setminus v)^*$ must form a nuclear complex,
	and the only disconnected nuclear complexes are of the form $\Delta_s \sqcup \Delta_t$.
	From this it follows that $(\mathcal{C}\setminus v)^*$ must be one of the following forms (without loss of generality)
	\begin{enumerate}
		\item $\Delta_m \sqcup \Delta_n$
		\item {$G^{m+n-m'-n'}(\Delta_{m'} \sqcup \Delta_{n'})$ with $-1 \le m'< m$ and $0 \le n' \le n$, but not both $m' = -1$ and $n' = n$}
		\item $G^{m+1}(\Delta_{n})$
		\item $G^{m+1}(\mathcal{D})$ where $\mathcal{D}$ is a nuclear complex on $N$ with more than one facet
		\item $G^{m+n+2}(\{\emptyset\})$
		\item $G^{m+n+2}(\{\})$.
	\end{enumerate}
	We handle each possibility separately.
	\\
	\indent
	{\bf Case 1.}  Assume $(\mathcal{C}\setminus v)^* = \Delta_m \sqcup \Delta_n$.
	Then $(\mathcal{C}\setminus v)^*$ has no ghost vertices, and so
	$\mathcal{C}\setminus v$ is not of Lawrence type.
	So, by Lemma \ref{minimality},
	$\mathcal{C} \setminus  v = D_{m,n}$ and so $\mathcal{C} = \cone^1(D_{m,n})$ with $v$ as the cone point.
	\\
	\indent
	{{\bf Case 2.} Assume $(\mathcal{C}\setminus v)^* = G^{m+n-m'-n'}(\Delta_{m'} \sqcup \Delta_{n'})$ with $-1\le m' < m$
	and $0 \le n' \le n$ but not both $m'=-1$ and $n' = n$.
	By Proposition \ref{m'n'}, $\mathcal{C}$ has a proper minor $M$ that is not nuclear.
	By induction, $M$ is not $\beta$-avoiding and so neither is $\mathcal{C}$.}
	\\
	\indent
	{\bf Case 3.} Assume $(\mathcal{C}\setminus v)^* = G^{m+1}(\Delta_n)$.
	Then Proposition \ref{n} implies that $\mathcal{C} = D_{m,n+1}$.
	\\
	\indent
	{\bf Case 4.} Assume $(\mathcal{C}\setminus v)^* = G^{m+1}(\mathcal{D})$ 
	where $\mathcal{D}$ is a nuclear complex on $N$ with more than one facet.
	By Proposition \ref{arbitrary}, $\mathcal{C}$ has a proper minor $M$ that is not nuclear.
	By induction, $M$ is not $\beta$-avoiding and so neither is $\mathcal{C}$.
	\\
	\indent
	{\bf Case 5.} Assume $(\mathcal{C}\setminus v)^* = G^{m+n+2}(\{\emptyset\})$.
	Then $\mathcal{C}\setminus v = \partial\Delta_{m+n+1}$.
	But $\link_v{\mathcal{C}} = D_{m,n}$ and so $v$ is not in any $m+n-1$ simplex with the vertices of $M$ and $N$.
	By Proposition \ref{ksimplex}, $\mathcal{C}$ is not $\beta$-avoiding.
	\\
	\indent
	{\bf Case 6.} Assume $(\mathcal{C}\setminus v)^* = G^{m+n+2}(\{\})$.
	Then $\mathcal{C}\setminus v = \Delta_{m+n+1}$
	and so $\mathcal{C} = \Lambda(D_{m,n})$ with $v$ as the added vertex.
\end{proof}

Now we have all the necessary tools to prove Theorem \ref{main}.

\begin{proof}[Proof of Theorem \ref{main}]
	By Proposition \ref{minor}, (\ref{unimodular}) implies (\ref{bavoiding}).
	Each nuclear complex can be obtained by applying the unimodularity-preserving operations
	from Section \ref{sec:basic} to $\Delta_n \sqcup \Delta_m$, which is unimodular by Proposition \ref{disjointsimplices}.
	Therefore (\ref{nuclear}) implies (\ref{unimodular}).
	We now show (\ref{bavoiding}) implies (\ref{nuclear}).
	\\
	\indent
	Let $\mathcal{C}$ be a $\beta$-avoiding complex on $q$ vertices.
	We show that $\mathcal{C}$ is nuclear by induction on $q$.
	For the base case, note that all simplicial complexes on $2$ or fewer vertices are both $\beta$-avoiding and nuclear.
	We may assume that $\mathcal{C}$ has no ghost vertices since adding and removing ghost vertices
	does not affect the properties of being nuclear nor $\beta$-avoiding.
	In light of Propositions \ref{square} and \ref{disconnected} and Lemma \ref{1skel},
	we only need to consider the cases where the $1$-skeleton of $\mathcal{C}$ is $K_q$,
	or the union of two complete graphs, glued along a common \emph{nonempty} clique.
	Therefore, we can choose a vertex $v$ of $\mathcal{C}$ that is in a facet with every other vertex.
	Since $\mathcal{C}$ is $\beta$-avoiding, so is $\link_v(\mathcal{C})$.
	By induction, $\link_v(\mathcal{C})$ is also nuclear.
	There are five main cases.
	\\
	\indent
	{\bf Case 1.}
	This is the case where $\link_v(\mathcal{C}) = \Lambda(\mathcal{D})$ for a nuclear $\mathcal{D}$.
	Let $F$ denote the big facet of $\mathcal{C}$ and let $u$ be the vertex of $\mathcal{C}$ that is not in $F$.
	Then $\mathcal{C} = \Lambda(\mathcal{C}\setminus u)$, and $\mathcal{C}\setminus u$ is nuclear by induction.
	\\
	\indent
	{\bf Case 2.}
	This is the case where $\link_v(\mathcal{C})$ has ghost vertices.
	Since $v$ connects to all other vertices, this is impossible.
	\\
	\indent
	{\bf Case 3.}
	This is the case where $\link_v(\mathcal{C}) = \cone^p(\Delta_m \sqcup \Delta_n)$ with $p,m,n \ge 0$.
	We split this into the sub cases, the first where $m,n \ge 1$ and the second where $m = 0$.
	\\
	\indent
	{\bf Case 3.1}
	This is the sub case where $m,n \ge 1$.
	Proposition \ref{linkdual} gives
	$\mathcal{C}^*\setminus v = \cone^p(D_{m,n})$ and so $\mathcal{C}^*$ has $D_{m,n}$ induced.
	Proposition \ref{iteratedcone} then implies that $\mathcal{C}^*$ is either $\cone^{p+1}(D_{m,n})$, or $\cone^p(D_{m,n})$
	with $v$ as a ghost vertex.
	In the first case, $\mathcal{C} = \cone^{p+1}(\Delta_m \sqcup \Delta_n)$.
	In the second case, $\mathcal{C} = \Lambda(\cone^{p}(\Delta_m\sqcup \Delta_n))$.
	\\
	\indent
	{\bf Case 3.2}
	This is the sub case where $m = 0$.
	Proposition \ref{linkdual} gives $\mathcal{C}^*\setminus v = \cone^p(D_{0,n})$ which has a ghost vertex $u$.
	Then $\mathcal{C}^*\setminus u$ is nuclear by induction and since $u$ is still a ghost vertex in $\mathcal{C}^*$,
	this implies $\mathcal{C}^*$ is nuclear.
	Proposition \ref{dualclosed} implies that $\mathcal{C}$ is nuclear.
	\\
	\indent
	{\bf Case 4.}
	This is the case where $\link_v(\mathcal{C}) = \cone^p(D_{m,n})$ with $p \ge 0$ and $m,n \ge 1$.
	By induction, all proper minors of $\mathcal{C}$ are nuclear.
	Then Lemma \ref{case4} implies that $\mathcal{C}$ is nuclear.
	\\
	\indent
	{\bf Case 5.} This is the case where $\link_v(\mathcal{C}) = \Delta_{k}$ with $k \ge -2$.
	Since $\link_v(\mathcal{C})$ has no ghost vertices $k = q-2$.
	So $\mathcal{C} = \Delta_{q-1}$.
\end{proof}

\section{Unimodularity of  Non-Binary Hierarchical Models}\label{sec:non}

In this section we consider the unimodularity of hierarchical models that are non-binary;
i.e. determining when $\mathcal{A}_{\calc,\bfd}$ is unimodular for arbitrary $\bfd$.
We begin by showing how Theorem \ref{main} makes this problem more tractable.
Unlike in the binary case, 
the Alexander dual does not always preserve unimodularity.
{For example, $\mathcal{A}_{\calc,\bfd}$ is unimodular when $\calc = \Delta_1 \sqcup \Delta_1$ and $d = (2,2,2,3)$.
However, $\mathcal{A}_{\calc^*,\bfd}$ is not (see Proposition \ref{badnonbinary}).}
Because of this, our proof of Proposition \ref{link} is not  valid
in the non-binary case.
The proposition statement does however generalize as we show with Corollary \ref{linknonbinary}.
Proposition \ref{badnonbinary} gives a list of known
minimal pairs $(\mathcal{C},\bfd)$ with $\calc$ nuclear where $\mathcal{A}_{\calc,\bfd}$ is not unimodular.
We then use this list to pose a simple question whose answer would be a step in the general classification.

\begin{prop}\label{conformal}
	Assume $\mathcal{A}_{\mathcal{C},{\bf d}}$ is unimodular.
	Then for all ${\bf d}'$ such that ${\bf d}' \le {\bf d}$ componentwise,
	$\mathcal{A}_{\mathcal{C},{\bf d}'}$ is unimodular.
\end{prop}
\begin{proof}
	Note that $\mathcal{A}_{\mathcal{C},{\bf d}'}$ can be realized as a subset of columns of $\mathcal{A}_{\mathcal{C},{\bf d}}$.
\end{proof}

The special case where ${\bf d'} = {\bf 2}$ gives us the following useful corollary.

\begin{cor}\label{nuclearnonbinary}
	If $\mathcal{A}_{\mathcal{C},{\bf d}}$ is unimodular, then $\mathcal{C}$ is nuclear.
\end{cor}

So in order to classify all unimodular hierarchical models we only need to consider nuclear complexes.
Therefore we can approach the general classification problem by looking at each nuclear $\mathcal{C}$
and identifying the minimal values of ${\bf d}$ that
give rise to non-unimodular $\mathcal{A}_{\mathcal{C},{\bf d}}$.
Before proceeding with this,
we show that the class of unimodular $\mathcal{A}_{\mathcal{C},{\bf d}}$ is still closed under taking minors of $\mathcal{C}$.
However Corollary \ref{dual_uni} fails in the non-binary case; we cannot freely take the Alexander dual of $\mathcal{C}$.
Because of this, our proof of Corollary \ref{link} is not valid in the non-binary case.
We will give an alternate proof that the class of unimodular $\mathcal{A}_{\mathcal{C},{\bf d}}$ is closed under
taking links in $\mathcal{C}$.
We start with a useful proposition.

\begin{prop}\label{projection}
	Let $A \in \rr^{r\times n}$ be a unimodular matrix with columns $\{{\bf a}_i\}_{i=1}^n$.
	Assume ${\bf a}_n$ is nonzero.
	Let $A'$ be the matrix that results when we project $A$ onto the hyperplane orthogonal to ${\bf a}_n$.
	Then $A'$ is unimodular.
\end{prop}
\begin{proof}
	We may assume $A$ has rank $r$ since otherwise we can delete unnecessary rows.
	Let $B \in \rr^{n\times (n-r)}$ have rank $n-r$ such that $AB^T = 0$.
	Denote the columns of $B$ as $\{{\bf b}_i\}_{i=1}^n$.
	By Proposition \ref{dualdet}, $B$ is also unimodular.
	Denote the columns of $A'$ as $\{{\bf a}_i'\}_{i=1}^{n-1}$.
	Then
	\[
		{\bf a}_i' = {\bf a}_i - \frac{\langle {\bf a}_i, {\bf a}_n \rangle}{\|{\bf a}_n\|^2} {\bf a}_n.
	\]
	Let $B'$ be the matrix with columns $\{b_i\}_{i=1}^{n-1}$.
	Then $B'$ has rank $n-r$ or $n-r-1$.
	The second case implies that ${\bf b}_n$ is a coloop in the matroid underlying $B$.
	Since the matroids underlying $A$ and $B$ are duals, this would imply that ${\bf a}_n = 0$.
	But then $A' = A$ which we assumed to be unimodular.
	So we can assume that $\rank(B') = \rank(B) = n-r$ and that ${\bf a}_n \neq 0$.
	Therefore $\rank(A') = r-1$ and so the dimension of its kernel is $n-r$, which is the rank of $B'$.
	We claim that $A'(B')^T = 0$.
	From this it follows by Proposition \ref{dualdet} that $A'$ is unimodular.
	\\
	\indent
	Now we prove the claim.
	Let ${\bf b}_i$ be a column of $B'$ (so $1 \le i \le n-1$).
	Letting $b_{ji}$ denote the $j$th entry of ${\bf b}_i$, we have
	\[
		A'{\bf b}_i^T = \sum_{j=1}^{n-1} b_{ji} \left({\bf a_j} - \frac{\langle {\bf a}_i, {\bf a}_n \rangle}{\|{\bf a}_n\|^2} {\bf a}_n\right).
	\]
	Note that $b_{jn}\left({\bf a}_n - \frac{\langle {\bf a}_n, {\bf a}_n \rangle}{\|{\bf a}_n\|^2} {\bf a}_n\right) = 0$,
	so we can add the $n$th term to the above sum.
	This enables us to break it up as follows
	\[
		\sum_{j=1}^n b_{ji}{\bf a}_j - \frac{1}{\|{\bf a}_n\|^2}\sum_{j=1}^n b_{ji} \langle {\bf a}_j, {\bf a}_n\rangle {\bf a}_n.
	\]
	The first term is $A{\bf b}_i^T = 0$.
	The second term is $\left(\frac{1}{\|{\bf a}_n\|^2} {\bf a}_n^TA{\bf b}_i^T\right)\cdot{\bf a}_n = 0$.
	So the claim is proven.
\end{proof}

\begin{cor}\label{linknonbinary}
	Assume $\mathcal{A}_{\mathcal{C},{\bf d}}$ is unimodular.
	Let $v$ be a vertex of $\mathcal{C}$ and let ${\bf d}'$ denote the vector obtained by deleting the entry for $v$ from ${\bf d}$.
	Then $\mathcal{A}_{\link_v{\mathcal{C}},{\bf d}'}$ is unimodular.
\end{cor}
\begin{proof}
	Let $A = \mathcal{A}_{\link_v{\mathcal{C}},{\bf d}'}$ and let $B = \mathcal{A}_{\mathcal{C}\setminus v, {\bf d}'}$.
	By the remark following Lemma 2.2 in \cite{HostenSullivant2007} and the lemma itself, we can write
	\[
		\mathcal{A}_{\mathcal{C},{\bf d}} =\left[\begin{matrix}
						A &  {\bf 0}  & \ldots & {\bf 0}\\
						{\bf 0}  &  A & \ldots & {\bf 0}\\
						\vdots & \vdots & \ddots & \vdots\\
						{\bf 0}  &   {\bf 0}       &\ldots & A\\
						B & B & \dots & B
					\end{matrix}\right].
	\]
	Assume $A \in \rr^{m \times n}$ and $B \in \rr^{l \times n}$ and so $\mathcal{A}_{\mathcal{C},{\bf d}} \in \rr^{dm+l, dn}$.
	{ By Proposition \ref{obvious}(\ref{item:rowspace}),
	we may remove rows from $\mathcal{A}$ to make $A,B,\mathcal{A}$ all have full row rank.
	So assume that they do.}
	Let $\mathcal{A}'$ denote the matrix that results when we project $\mathcal{A}_{\mathcal{C},{\bf d}}$
	onto the subspace orthogonal to the last $(d-1)n$ columns.
	Then if $1 \le i \le n$, the $i$th column of $\mathcal{A}'$ can be expressed as the $i$th column of $\mathcal{A}_{\mathcal{C},{\bf d}}$
	minus a linear combination of the last $(d-1)n$ columns of $\mathcal{A}_{\mathcal{C},{\bf d}}$,
	all of which are $0$ in the top $m$ rows.
	So this means that the top $m$ rows of $\mathcal{A}'$ are
	\[
		\begin{pmatrix} A & {\bf 0} & \dots & {\bf 0} \end{pmatrix}.
	\]
	Furthermore, since $A$ and $\mathcal{A}_{\mathcal{C},{\bf d}}$ both have full row rank,
	the final $(d-1)n$ columns of $\mathcal{A}_{\mathcal{C},{\bf d}}$ have rank $(d-1)m+l$.
	Therefore $\mathcal{A}'$ has rank $m$.
	Since $A$ also has rank $m$, we may delete the bottom $(d-1)m+l$ rows of $\mathcal{A}'$ without affecting the rowspace
	and therefore unimodularity.
	So the matrix $\begin{pmatrix} A & {\bf 0} & \dots & {\bf 0} \end{pmatrix}$, and therefore $A$, is unimodular.
\end{proof}

Our proof of Proposition \ref{induced} is still valid in the non-binary case.
Therefore, we have the following.

\begin{prop}\label{minornonbinary}
	Assume $\mathcal{C}'$ is a minor of $\mathcal{C}$.
	If $\mathcal{A}_{\mathcal{C},{\bf d}}$ is unimodular, then so is $\mathcal{A}_{\mathcal{C}',{\bf d}'}$
	where $\bfd'$ is the restriction of $\bfd$ to the vertices that are in $\calc'$.
\end{prop}

The proofs of Propositions \ref{cone}, \ref{ghost} generalize to the non-binary setting and so we have the following.

\begin{prop}\label{coneghost}
	If the pair $(\mathcal{C},{\bf d})$ gives rise to unimodular $\mathcal{A}_{\mathcal{C},{\bf d}}$,
	so do $(\cone^p(\mathcal{C}),{\bf d}')$ and $(G(\mathcal{C}),{\bf d}'')$
	where $\bfd' = \begin{pmatrix} \bfd & c_1 & \dots & c_p \end{pmatrix}$ and
	$\bfd'' = \begin{pmatrix} \bfd & c \end{pmatrix}$ for any $c,c_1,\dots,c_p \ge 2$.
\end{prop}

\begin{proof}
In the case of a cone, the matrix $\mathcal{A}_{\cone^p(\mathcal{C}),{\bf d}'}$
is a block diagonal matrix:
$$
\mathcal{A}_{\cone^p(\mathcal{C}),{\bf d}'} = 
\begin{pmatrix}
\mathcal{A}_{\mathcal{C}, {\bf d}} & 0 & \cdots & 0  \\
0 &  \mathcal{A}_{\mathcal{C}, {\bf d}} & \cdots & 0  \\
\vdots & \vdots & \ddots & \vdots  \\
0 & 0 & \cdots & \mathcal{A}_{\mathcal{C}, {\bf d}}
\end{pmatrix},
$$
with $c_{1}c_{2} \cdots c_{p}$ blocks down the diagonal.
In the case of adding a ghost vertex, the matrix $\mathcal{A}_{G(\mathcal{C}),{\bf d}'}$ has repeated columns:
$$
\mathcal{A}_{G(\mathcal{C}),{\bf d}''} = 
\begin{pmatrix}
\mathcal{A}_{\mathcal{C}, {\bf d}} & \mathcal{A}_{\mathcal{C}, {\bf d}} & \cdots & \mathcal{A}_{\mathcal{C}, {\bf d}}  \\
\end{pmatrix}
$$
with $c$ blocks.
\end{proof}

We can also extend Proposition \ref{disjointsimplices}.

\begin{prop}\label{disjointnonbinary}
	Let $\mathcal{C} = \Delta_m \sqcup \Delta_n$ and ${\bf d} \in \zz_{\ge 2}^{m+n+2}$.
	Let $M,N$ denote the vertex sets of $\mathcal{C}$ in the $\Delta_m,\Delta_n$ respectively.
	Let $\mathcal{D}$ be the complex with facets $\{1,2\}$ and let ${\bf e} = (e_1,e_2)$ where
	\[
		e_1 = \prod_{v \in M} d_v \qquad e_2 = \prod_{v \in N}d_v.
	\]
	Then $\mathcal{A}_{\mathcal{C},{\bf d}} = \mathcal{A}_{\mathcal{D},{\bf e}}$.
	This matrix is unimodular.
\end{prop}
\begin{proof}
	We can see that $\mathcal{A}_{\mathcal{C},{\bf d}} = \mathcal{A}_{\mathcal{D},{\bf e}}$ by inspection.
	The matrix $\mathcal{A}_{\mathcal{D},(e_1,e_2)}$ is the vertex edge incidence matrix of a complete bipartite
	graph with $e_1$ and $e_2$ vertices in each set of the partition.
	Such vertex-edge incidence matrices are
	examples of network matrices and are hence totally unimodular \cite[Ch.~19]{Schrijver1986}.
\end{proof}

We also note that Proposition \ref{lawrence} holds in a slightly more general setting.
\begin{prop}\label{lawrencenonbinary}
	Let $(\mathcal{C},{\bf d})$ be such that $\mathcal{A}_{\mathcal{C},{\bf d}}$ is unimodular.
	Then $\mathcal{A}_{\Lambda(\mathcal{C}),{\bf d}'}$ is also unimodular if ${\bf d}' = \begin{pmatrix} \bfd & 2 \end{pmatrix}$.
\end{prop}
\begin{proof}
	The proof here is similiar to the proof for Proposition \ref{lawrence}.
	Note that $\Lambda(\mathcal{A}_{\calc,\bfd}) = \mathcal{A}_{\Lambda(\calc),\bfd'}$.
	Hence Theorem 7.1 in \cite{sturmfels} implies the proposition.
\end{proof}

Propositions \ref{conformal} and \ref{minornonbinary} suggest that
we might be able to use a short list of pairs $(\mathcal{C},{\bf d})$ that give rise
to non-unimodular $\mathcal{A}_{\mathcal{C},{\bf d}}$ to eliminate the remaining non-unimodular pairs.
The following proposition lists all non-binary minimally non-unimodular complexes
that we presently know.

\begin{prop}\label{badnonbinary}
	The following pairs give rise to non-unimodular $\mathcal{A}_{\mathcal{C},{\bf d}}$
	\begin{enumerate}
		\item\label{item:triangle} $\mathcal{C} = \Lambda(\Delta_0 \sqcup \Delta_0)$ with facets $\{12,23,13\}$ and ${\bf d} = (3,3,3)$
		\item\label{item:square} $\mathcal{C} = D_{1,1}$ with facets $\{12,23,34,14\}$ and ${\bf d} = (2,2,2,3)$
		\item\label{item:other} $\mathcal{C} = \Lambda(G(\Delta_0 \sqcup \Delta_0))$ with facets $\{12,13,234\}$ and ${\bf d} = (4,2,2,2)$
		\item\label{item:llsquare} $\mathcal{C} = \Lambda(D_{1,1})$ with facets $\{1234,125,235,345,145\}$ and ${\bf d} = (2,2,2,2,3)$.
	\end{enumerate}
\end{prop}
\begin{proof}
	By Table 1 in \cite{ohsugi-hibi2010}, (\ref{item:triangle}) is not unimodular.
	Computing a Graver basis with 4ti2 \cite{4ti2} shows that (\ref{item:square}) and (\ref{item:llsquare}) are not unimodular.  
	{Computations are available on our website \cite{BernsteinWeb2016}.}
	If $\mathcal{C}$ and ${\bf d}$ are as in (\ref{item:other}) then $\mathcal{A}_{\mathcal{C},{\bf d}} = \mathcal{A}_{J_1^*,{\bf 2}}$
	which is not unimodular by Proposition \ref{prop:badcomplexes}.
\end{proof}

From these non-unimodular examples and the constructive results earlier
in the section, we can classify all unimodular hierarchical models
that are nuclear with nucleus $D_{m,n}$, $m,n \geq 1$.

\begin{thm}\label{thm:dmnnonbinary}
Let $\calc$ be a nuclear complex with nucleus $D_{m,n}$ with
$m,n \geq 1$.  Then $\mathcal{A}_{\calc, {\bf d}}$ is unimodular
if and only if
\begin{enumerate}
\item  $d_{v} =2$ for all $v$ in the original $D_{m,n}$ and
\item  $d_{v} = 2$ for all $v$ corresponding to a Lawrence lifting.
\end{enumerate}
\end{thm}

\begin{proof}
The fact that all such $\mathcal{A}_{\calc, {\bf d}}$
are unimodular follows from applying Propositions
\ref{coneghost}  and \ref{lawrencenonbinary}.  We now show that
these are the only unimodular $\mathcal{A}_{\calc, {\bf d}}$
with this type of complex $\calc$.

Note that since taking cone vertices commutes with adding ghost vertices
and Lawrence liftings and this always preserves unimodularity,
we can assume there are no cone vertices.  
Now if $\calc$ is a nuclear complex obtained by successively adding
ghost vertices and Lawrence vertices to $D_{m,n}$, then the
induced subcomplex on just the $D_{m,n}$ and Lawrence vertices
is an iterated Lawrence lifting of $D_{m,n}$.  If any
of the vertices of $v$ that are Lawrence vertices have $d_{v} > 2$,
then by taking suitable links, one obtains the complex
isomorphic to item (4) in Proposition \ref{badnonbinary}.
On the other hand, if $d_{v} > 2$ for some $v $ in the original
$D_{m,n}$, then by taking suitable links, one obtains the complex
isomorphic to item (2) in  Proposition \ref{badnonbinary}.
In summary, if the conditions of Theorem \ref{thm:dmnnonbinary}
are not satisfied, there is a minor of $\calc$ that is not unimodular.
\end{proof}

It remains to consider the unimodularity of $\mathcal{A}_{\calc,\bfd}$
when $\calc$ is a nuclear complex with nucleus $\Delta_{m} \sqcup \Delta_{n}$.
We do not have a general characterization in this case.
By Propositions \ref{coneghost} and \ref{commutes}, we may assume that $\calc$ has no cone vertices.  By Proposition \ref{disjointnonbinary},
we may assume that $m, n = 0$.
Note that $\Lambda( \Delta_{0} \sqcup \Delta_{0}) = \partial\Delta_2$, and
our results above (or \cite{ohsugi-hibi2010}) show that this
is unimodular if and only if some $d_{v} = 2$.  Similarly, the iterated 
Lawrence liftings $\Lambda^{p}( \Delta_{0} \sqcup \Delta_{0}) = 
\partial\Delta_{p+1}$ is unimodular if and only if at most two 
of the $d_{v}$ are greater than $2$.

We begin moving into uncharted territory when we consider cases where ghost 
vertices are involved. 
The simplest case to consider is the complex $\Lambda(G( \Delta_{0} \sqcup \Delta_{0}))$ which is the complex on $\{1,2,3,4\}$ with facets $\{12,13,234\}$.
Here $1$ is the Lawrence vertex and $4$ is the ghost vertex.
By Propositions \ref{minornonbinary} and \ref{badnonbinary} (\ref{item:other}),
we must have $d_1 \le 3$.
If $d_1 = 2$ then Proposition \ref{lawrencenonbinary} implies that $\mathcal{A}_{\mathcal{C},{\bf d}}$ will be unimodular.
So we only need to consider the case where $d_1 = 3$.
Propositions \ref{conformal}, \ref{minornonbinary}, and \ref{badnonbinary} (\ref{item:triangle}) imply that
one of $d_2$ or $d_3$ is $2$, so assume $d_2 = 2$.
This leaves us with the following question, which would classify
unimodularity for $\calc = \Lambda(G( \Delta_{0} \sqcup \Delta_{0}))$.

\begin{ques}
	Let $\mathcal{C} = \Lambda(G(\Delta_0 \sqcup \Delta_0))$ be the complex on $\{1,2,3,4\}$ with facets $\{12,13,234\}$.
	Let ${\bf d} = (3,2,d_{3},d_{4})$.
	For what values of $d_{3},d_{4}$ is the matrix $\mathcal{A}_{\mathcal{C},{\bf d}}$ unimodular?
\end{ques}

\section*{Acknowledgments}
Daniel Irving Bernstein was partially supported by the US National Science Foundation (DMS 0954865).
Seth Sullivant was partially supported by the David and Lucille Packard 
Foundation and the US National Science Foundation (DMS 0954865).

\bibliography{unimodular}
\bibliographystyle{plain}

\end{document}